\documentclass[a4paper]{amsart}

\usepackage{microtype}
\usepackage[centertags]{amsmath}

\usepackage{amsfonts}
\usepackage{amsthm}
\usepackage{newlfont}
\usepackage{amscd}
\usepackage{amsmath}
\usepackage{enumerate}
\usepackage[all,2cell]{xy}
\UseAllTwocells
\input xy
\xyoption{2cell}
\xyoption{all}
\usepackage{verbatim}
\usepackage{eucal}
\usepackage{enumitem}

\calclayout
\usepackage{xcolor}
\definecolor{darkblue}{rgb}{0,0,0.4}
\usepackage[ocgcolorlinks,colorlinks=true, citecolor=darkblue, filecolor=darkblue, linkcolor=darkblue, urlcolor=darkblue]{hyperref}
\usepackage{graphicx}

\newcommand{\DD}{\mathbb{D}}
\newcommand{\PP}{\mathbb{P}}
\newcommand{\TT}{\mathbb{T}}
\newcommand{\NN}{\mathbb{N}}

\newcommand{\CC}{\mathbb{C}}
\newcommand{\ZZ}{\mathbb{Z}}

\newcommand{\EE}{\mathbb{E}}

\def\X{\mathcal{X}}
\def\A{\mathcal{A}}

\def\QQ{\mathbb{Q}}
\def\C{\mathcal{C}}
\def\D{\mathcal{D}}

\def\M{\mathcal{M}}

\def\W{\mathcal{W}}
\def\I{\mathcal{I}}
\def\cS{\mathcal{S}}
\def\T{\mathcal{T}}
\def\P{\mathcal{P}}
\def\Q{\mathcal{Q}}
\def\J{\mathcal{J}}

\renewcommand{\L}{\mathcal{L}}
\def\U{\mathcal{U}}
\def\V{\mathcal{V}}
\def\W{\mathcal{W}}
\def\O{\mathcal{O}}

\def\E{\mathcal{E}}

\def\rH{\mathrm{H}}
\def\K{\mathcal{K}}

\def\C{\mathcal{C}}

\DeclareMathOperator{\Exc}{Exc}

\def\alp{{\alpha}}

\def\gam{{\gamma}}
\def\del{\delta}
\def\eps{{\varepsilon}}

\def\Om{{\Omega}}

\def\Del{{\Delta}}

\def\vphi{\varphi}

\newcommand{\HSwarrow}{\kern0.05ex\vcenter{\hbox{\Huge\ensuremath{\Swarrow}}}\kern0.05ex}
\newcommand{\hSwarrow}{\kern0.05ex\vcenter{\hbox{\huge\ensuremath{\Swarrow}}}\kern0.05ex}
\newcommand{\LLSwarrow}{\kern0.05ex\vcenter{\hbox{\LARGE\ensuremath{\Swarrow}}}\kern0.05ex}
\newcommand{\LSwarrow}{\kern0.05ex\vcenter{\hbox{\Large\ensuremath{\Swarrow}}}\kern0.05ex}

\newcommand{\HSearrow}{\kern0.05ex\vcenter{\hbox{\Huge\ensuremath{\Searrow}}}\kern0.05ex}
\newcommand{\hSearrow}{\kern0.05ex\vcenter{\hbox{\huge\ensuremath{\Searrow}}}\kern0.05ex}
\newcommand{\LLSearrow}{\kern0.05ex\vcenter{\hbox{\LARGE\ensuremath{\Searrow}}}\kern0.05ex}
\newcommand{\LSearrow}{\kern0.05ex\vcenter{\hbox{\Large\ensuremath{\Searrow}}}\kern0.05ex}

\newcommand{\HDownarrow}{\kern0.05ex\vcenter{\hbox{\Huge\ensuremath{\Downarrow}}}\kern0.05ex}
\newcommand{\hDownarrow}{\kern0.05ex\vcenter{\hbox{\huge\ensuremath{\Downarrow}}}\kern0.05ex}
\newcommand{\LLDownarrow}{\kern0.05ex\vcenter{\hbox{\LARGE\ensuremath{\Downarrow}}}\kern0.05ex}
\newcommand{\LDownarrow}{\kern0.05ex\vcenter{\hbox{\Large\ensuremath{\Downarrow}}}\kern0.05ex}

\newcommand{\HUparrow}{\kern0.05ex\vcenter{\hbox{\Huge\ensuremath{\Uparrow}}}\kern0.05ex}
\newcommand{\hUparrow}{\kern0.05ex\vcenter{\hbox{\huge\ensuremath{\Uparrow}}}\kern0.05ex}
\newcommand{\LLUparrow}{\kern0.05ex\vcenter{\hbox{\LARGE\ensuremath{\Uparrow}}}\kern0.05ex}
\newcommand{\LUparrow}{\kern0.05ex\vcenter{\hbox{\Large\ensuremath{\Uparrow}}}\kern0.05ex}
\newtheorem{thm}{Theorem}
\newtheorem{cor}[thm]{Corollary}
\newtheorem{lem}[thm]{Lemma}
\newtheorem{pro}[thm]{Proposition}

\newtheorem{expro}[thm]{Example/Proposition}

\numberwithin{thm}{section}
\numberwithin{equation}{section}

\theoremstyle{definition}
\newtheorem{define}[thm]{Definition}

\newtheorem{example}[thm]{Example}

\newtheorem{defn}[thm]{Definition}
\newtheorem{cons}[thm]{Construction}

\newtheorem{obs}[thm]{Observation}

\theoremstyle{remark}
\newtheorem{rem}[thm]{Remark}

\newtheorem{warning}[thm]{Warning}

\setenumerate{leftmargin=*}
\setitemize{leftmargin=*}

\DeclareMathOperator{\colim}{colim}

\DeclareMathOperator{\id}{id}
\DeclareFontFamily{OT1}{pzc}{}
\DeclareFontShape{OT1}{pzc}{m}{it}{<-> s * [1.10] pzcmi7t}{}
\DeclareMathAlphabet{\mathpzc}{OT1}{pzc}{m}{it}
\DeclareMathOperator{\cof}{cof}
\DeclareMathOperator{\fib}{fib}

\DeclareMathOperator{\Alg}{Alg}

\DeclareMathOperator{\Grp}{Grp}

\DeclareMathOperator{\op}{op}
\DeclareMathOperator{\Map}{Map}
\DeclareMathOperator{\red}{red}

\DeclareMathOperator{\Cat}{Cat}

\DeclareMathOperator{\CAlg}{CAlg}
\DeclareMathOperator{\Mod}{Mod}

\DeclareMathOperator{\Fun}{Fun}

\DeclareMathOperator{\Ob}{Ob}
\DeclareMathOperator{\Sp}{Sp}
\DeclareMathOperator{\Ho}{ho}

\DeclareMathOperator{\ev}{ev}
\DeclareMathOperator{\fin}{fin}
\DeclareMathOperator{\exc}{exc}

\DeclareMathOperator{\SqZ}{SqZ}

\DeclareMathOperator{\Ab}{Ab}

\DeclareMathOperator{\ab}{ab}

\DeclareMathOperator{\sSet}{sSet}
\DeclareMathOperator{\Loc}{Loc}

\DeclareMathOperator{\Mon}{Mon}

\DeclareMathOperator{\cosk}{cosk}

\DeclareMathOperator{\Fin}{Fin}
\DeclareMathOperator{\lev}{lev}

\newcommand{\Tow}{\operatorname{Tow}_{\mathrm{Pos}}}
\newcommand{\ET}{\E}

\def\x{\overset}

\def\MCom{\M\mathrm{Com}}

\newcommand{\tgpd}{\kern0.05ex\vcenter{\hbox{\footnotesize\ensuremath{2}}}\kern0.05ex\mathcal{G}pd}

\def\rar{\rightarrow}

\def\lrar{\longrightarrow}

\newcommand{\adj}{\mathrel{\substack{\longrightarrow \\[-.6ex] \x{\perp}{\longleftarrow}}}}

\newcommand{\<}{\left<}
\renewcommand{\>}{\right>}

\makeatletter
\def\maketag@@@#1{\hbox{\m@th\normalfont\normalsize#1}}
\makeatother

\title{On $k$-invariants for $(\infty, n)$-categories}

\author{Yonatan Harpaz}
\email{harpaz@math.univ-paris13.fr}
\address{LAGA\\ Institut Galilée\\
99 Avenue Jean Baptiste Cl\'ement\\
93430 Villetaneuse, France.}
\author{Joost Nuiten}
\email{joost-jakob.nuiten@umontpellier.fr}
\address{IMAG\\ Universit\'e de Montpellier\\
Place Eug\`ene Bataillon\\
34090 Montpellier, France.}
\author{Matan Prasma}
\email{mtnprsm@gmail.com}
\address{Neurocat GmbH\\
Rudower Chaussee 29\\ 12489 Berlin, Germany}

\begin{document}

\begin{abstract}
Every $(\infty, n)$-category can be approximated by its tower of homotopy $(m, n)$-categories. In this paper, we prove that the successive stages of this tower are classified by $k$-invariants, analogously to the classical Postnikov tower for spaces. Our proof relies on an abstract analysis of Postnikov-type towers equipped with $k$-invariants, and also yields a construction of $k$-invariants for algebras over $\infty$-operads and enriched $\infty$-categories.
\end{abstract}

\maketitle

\tableofcontents

\section{Introduction}
The weak homotopy type of a topological space can be conveniently studied using its Postnikov tower
$$\xymatrix{
X\ar[r] & \dots\ar[r] & \tau_{\leq a}X\ar[r] & \tau_{\leq a-1}X\ar[r] & \dots \ar[r] & \tau_{\leq 0}X=\pi_0(X).
}$$
The Postnikov tower allows one (theoretically) to reconstruct $X$ from algebraic and cohomological data. Indeed, the lowest stages of this tower encode the path components of $X$ and its fundamental groupoid. For the higher stages, the passage from $\tau_{\leq a-1}X$ to $\tau_{\leq a}X$ is completely determined by a cohomology class
$$
k_{a}\in \rH^{a+1}\big(\tau_{\leq a-1}X, \pi_{a}(X)\big).
$$
Indeed, given a map $f\colon Y\lrar \tau_{\leq a-1}X$, there exists a lift
\begin{equation}\label{d:lift}\vcenter{\xymatrix{
& \tau_{\leq a}X\ar[d]\\
Y\ar@{..>}[ru]\ar[r]_-f & \tau_{\leq a-1}X
}}\end{equation}
if and only if the cohomology class $f^*k_{a}$ vanishes on $Y$. In this case, the $i$-th homotopy group of the space of lifts \eqref{d:lift} can be identified (noncanonically) with the $(a-i)$-th cohomology group of $Y$ with coefficients in $f^*\pi_{a}(X)$. Here it should be noted that the homotopy groups $\pi_{a}(X)$ typically form a \textbf{local system of abelian groups}.

The purpose of this paper is to describe an analogue of the Postnikov tower for $(\infty, n)$-categories. More precisely, every $(\infty, n)$-category $\C$ admits a tower of homotopy $(m, n)$-categories \cite[\S 3.5]{Lur09b} (see Section \ref{s:local systems})
$$\xymatrix{
\C \ar[r] & \dots\ar[r] & \Ho_{(m, n)}\C\ar[r] & \Ho_{(m-1, n)}\C \ar[r] & \dots \ar[r] & \Ho_{(n, n)}\C.
}$$
Our main result asserts that there are again cohomology classes which control the passage from the homotopy $(m, n)$-category to the homotopy $(m+1, n)$-category:
\begin{thm}[informal]\label{t:main-theorem-intro}
For each $a\geq 2$, the extension $\Ho_{(n+a, n)}\C\lrar \Ho_{(n+a-1, n)}\C$ is classified by a $k$-invariant
$$
k_{a}\in \rH^{a+1}\big(\Ho_{(n+a-1, n)}\C, \pi_a(\C)\big),
$$
where $\pi_a(\C)$ is a \textbf{local system of abelian groups} on the $(\infty, n)$-category $\Ho_{(n+1, n)}\C$.
\end{thm}
In the case of $(\infty, 1)$-categories, these $k$-invariants have also been constructed explicitly in terms of simplicial categories by Dwyer--Kan--Smith \cite{DKS86, DK88}. For $n>1$, the above result is stated (without proof) and used in \cite{Lur09b}. In \cite{part3}, we have used this result as part of an obstruction-theoretic proof of the fact that adjunctions in $(\infty, 2)$-categories are uniquely determined at the level of the homotopy $2$-category (cf.\ also \cite{RV16}).

To make Theorem \ref{t:main-theorem-intro} more precise, let us recall that for any local system of abelian groups $\A$ on a space $X$, there exist Eilenberg--Maclane spaces $\mathrm{K}(\A, a)\lrar \tau_{\leq 1}X$, defined in the homotopy category $\Ho(\cS_{/\tau_{\leq 1}X})$ by the following universal property: for every map $f\colon Y\lrar \tau_{\leq 1}X$, there is a natural bijection
$$
\rH^a(Y, f^*\A) \cong \pi_0\Map_{/\tau_{\leq 1}(X)}\big(Y, \mathrm{K}(\A, a)\big).
$$
In fact, the Eilenberg--Maclane spaces $\mathrm{K}(\A, a)$ are related by equivalences 
$$
\mathrm{K}(\A, a)\stackrel{\sim}{\lrar} \Omega_{/\tau_{\leq 1}X}\mathrm{K}(\A, a+1)
$$
where $\Omega_{/\tau_{\leq 1}X}\mathrm{K}(\A, a+1)$ computes the fiberwise loop space of $\mathrm{K}(\A, a+1)$ over $\tau_{\leq 1}X$ (at the basepoints given by the canonical section classifying the zero cohomology class). In other words, these Eilenberg--Maclane spaces can be organized into a parametrized spectrum $\rH\A$ over $\tau_{\leq 1}X$ such that $\mathrm{K}(\A, a)\simeq \Omega^{\infty}\big(\Sigma^a\rH\A\big)$ \cite{MS06}. From an $\infty$-categorical perspective, this parametrized spectrum can also be described more precisely as follows \cite{ABGHR14}: the local system $\A$ determines a functor of $\infty$-categories $\rH\A\colon \tau_{\leq 1}X\lrar \Ab\lrar \Sp$ sending each $x\in \tau_{\leq 1}X$ to the Eilenberg--Maclane spectrum of the abelian group $\A_x$. By the Grothendieck construction, such an $\infty$-functor to spectra can equivalently be viewed as a spectrum object in spaces over $\tau_{\leq 1}X$.

In these terms, the $k$-invariants can be interpreted as maps that fit into commuting squares for $a\geq 2$
$$\xymatrix{
\tau_{\leq a}X\ar[d]\ar[r] & \tau_{\leq 1}X\ar[d]^0\\
\tau_{\leq a-1}X\ar[r]_-{k_{a}} & \Omega^{\infty}\big(\Sigma^{a+1}\rH\pi_{a}(X)\big).
}$$
Here the right vertical map classifies the zero cohomology class. In fact, this square is homotopy Cartesian, which implies that the space of sections \eqref{d:lift} is homotopy equivalent to the space of null-homotopies of $f^*k_{a}$.

Our more precise version of Theorem \ref{t:main-theorem-intro} is then the following:
\begin{thm}[Theorem \ref{t:main-theorem-oo-n-cats}]\label{thm:main-theorem-intro-more-precise}
For any $(\infty, n)$-category $\C$ and $a\geq 2$, there is a parametrized spectrum object $\rH\pi_a(\C)$ internal to $(\infty, n)$-categories, whose base object is $\Ho_{(n+1, n)}\C$, so that there is a pullback square of $(\infty, n)$-categories
\begin{equation}\label{d:homotopy category as square zero}\vcenter{\xymatrix{
\Ho_{(n+a, n)}\C\ar[d]\ar[r] & \Ho_{(n+1, n)}\C\ar[d]^0\\
\Ho_{(n+a-1, n)}\C\ar[r]_-{k_{a}} & \Omega^{\infty}\big(\Sigma^{a+1}\rH\pi_a(\C)\big).
}}\end{equation}
\end{thm}
Furthermore, we prove that the parametrized spectrum $\rH\pi_a(\C)$ can indeed be thought of as an Eilenberg--Maclane spectrum: it is contained in the heart of a certain $t$-structure on the $\infty$-category of parametrized spectrum objects over $\Ho_{(n+1, n)}\C$ (Corollary \ref{c:local systems}). This heart consists of local systems of abelian groups on the $(\infty, n)$-category $\Ho_{(n+1, n)}\C$, as defined (somewhat informally) in \cite{Lur09b} (see Definition \ref{d:local systems} and Remark \ref{r:local system}). 

To prove Theorem \ref{thm:main-theorem-intro-more-precise}, the main idea will be to proceed by induction on the categorical dimension $n$. More precisely, the structure of the Postnikov tower, together with its $k$-invariants, can be axiomatized in terms of `abstract Postnikov towers'. We prove that a functorial construction of such towers in a symmetric monoidal $\infty$-category $\V$ gives rise to an abstract Postnikov tower construction on the $\infty$-category $\Cat(\V)$ of $\V$-enriched $\infty$-categories (Theorem \ref{t:main-theorem}). Furthermore, the resulting functorial abstract Postnikov tower on $\Cat(\V)$ respects the natural symmetric monoidal structure on $\Cat(\V)$ inherited from $\V$. This can be used to proceed inductively. 

More generally, this argument can also be used to provide $k$-invariants for Postnikov towers of algebras over $\infty$-operads (see Proposition \ref{p:quillen-for-algebras} and Example \ref{e:algebras in prestable}). These $k$-invariants typically take values in certain Andr\'e-Quillen cohomology groups, and have also been considered (in specific cases) in \cite{GH00, BM, Lur14}.

\subsection*{Outline} Let us now give an outline of this work: in Section \ref{s:tangentmonoidal}, we recall the definition of the tangent bundle of an $\infty$-category and the related theory of `square zero extensions'. Furthermore, we discuss the `square zero' monoidal structure on the tangent bundle of a symmetric monoidal presentable $\infty$-category $\C$; this is particularly useful to describe the tangent bundles to categories of algebras.

In Section \ref{s:system}, we give an abstract axiomatization of towers of square zero extensions, which we call \textbf{abstract Postnikov towers}, as well as multiplicative refinements thereof. In particular, we show how multiplicative abstract Postnikov towers induce (multiplicative) abstract Postnikov towers for algebras over $\infty$-operads. As the basis of our inductive proof, we show that the Postnikov tower of spaces is part of a multiplicative functorial abstract Postnikov tower. Section \ref{s:ncat} contains our main result, Theorem \ref{t:main-theorem}: we show that multiplicative abstract Postnikov towers induce multiplicative abstract Postnikov towers at the level of enriched $\infty$-categories.

In Section \ref{s:local systems}, we apply this result inductively to prove that the homotopy $(m, n)$-categories of an $(\infty, n)$-category are part of a multiplicative abstract Postnikov tower (Theorem \ref{t:main-theorem-oo-n-cats}); in particular, this provides the required pullback squares \eqref{d:homotopy category as square zero}. Finally, we discuss how the tangent bundle of $(\infty, n)$-categories carries a (family of) $t$-structures, whose heart consists of the category of local systems of abelian groups on $(\infty, n)$-categories (Definition \ref{d:local systems}). The parametrized spectra $\rH\pi_a(\C)$ appearing in \eqref{d:homotopy category as square zero} then appear as the Eilenberg--Maclane spectra associated to such local systems.

\subsection*{Conventions}
We will make use of the language of $\infty$-categories, i.e.\ quasicategories, following the standard references \cite{Lur09, Lur14}.
Furthermore, we will employ the following terminology: we will refer to symmetric monoidal $\infty$-categories as \textbf{SM $\infty$-categories}. A presentable SM $\infty$-category is a presentable $\infty$-category equipped with a closed symmetric monoidal structure, i.e.\ an object in $\CAlg(\Pr^\mathrm{L})$.

\subsection*{Acknowledgments}
This project was funded by the CNRS, under the programme Projet Exploratoire de Premier Soutien ``Jeune chercheuse, jeune chercheur'' (PEPS JCJC). Furthermore, J.N.\ has received funding from the European Research Council (ERC) under the European Union’s Horizon 2020 research and innovation programme (grant agreement No 768679). M.P.\ was supported by grant SFB 1085.

\section{Recollections on square zero extensions}\label{s:tangentmonoidal}
The purpose of this section is to recall some elements of the cotangent complex formalism described by Lurie \cite[\S 7.3]{Lur14}. In particular, we will recall the definition of the tangent bundle of an $\infty$-category $\C$ and the notion of a square zero extension. To motivate this terminology, we show that the tangent bundle inherits a `square zero' monoidal structure from $\V$.

\subsection{Tangent bundle}
Let $\V$ be an $\infty$-category with finite limits. Following Lurie \cite[Definition 7.3.1.9]{Lur14}, we define the \textbf{tangent bundle} of $\V$ to be the $\infty$-category
$$
\T\V = \Exc(\cS^{\fin}_\ast,\V)
$$
of excisive functors $F\colon \cS^{\fin}_\ast\lrar \V$ from the $\infty$-category of finite pointed spaces, i.e.\ those functors sending pushout squares to pullback squares. The $\infty$-category $\T\V$ comes with a projection
$$\xymatrix{
\pi= \ev_\ast\colon \T\V\ar[r] & \V & & \Omega^\infty= \ev_{S^0}\colon \T\V\ar[r] & \V
}$$
the functors taking the base, resp.\ the (parametrized) infinite loop space object underlying such a parametrized spectrum. The functor $\pi$ is a Cartesian fibration and admits both a left and a right adjoint, both taking the constant excisive functor on an object in $\V$. We refer to the fiber of $\pi$ at an object $X\in \V$ as the \textbf{tangent $\infty$-category} $\T_X\V$ of $\V$ at $X$; it can be identified with the stabilization $\Sp(\V_{/X})$ of the over-category $\V_{/X}$ \cite[\S 7.3.1]{Lur14}. When $\V$ is presentable, $\T\V$ and each of the fibers $\T_X\V$ are presentable as well and the functor $\Om^\infty$ admits a left adjoint $\Sigma^\infty_+$.

\begin{example}\label{e:stable}
Let $\V$ be an $\infty$-category with finite limits and consider the full subcategory $\mathrm{Ret}\subseteq \cS^{\fin}_\ast$ on $\ast$ and $S^0$. Then $\mathrm{Ret}$ is equivalent to the retract category \cite[Definition 4.4.5.2]{Lur09} and there are functors
$$\xymatrix{
\T\V\ar[r]^-G\ar[rd]_{\ev_*} & \Fun(\mathrm{Ret}, \V)\ar[d]_{\ev_*}\ar[r]^-{\fib} & \V\times \V\ar[ld]^{\pi_1}\\
& \V &
}$$
Here the first horizontal functor is given by restriction and $\fib$ sends a retract diagram $X\lrar Y\lrar X$ to the tuple $\big(X, \fib(Y\lrar X)\big)$. The first functor exhibits $\T\V$ as the fiberwise stabilization of $\Fun(\mathrm{Ret}, \V)$, that is, for every $X \in \V$ the induced functor $\T_X\V \lrar \Fun(\mathrm{Ret}, \V)_X \simeq (\V_{/X})_*$ on fibers over $X$ exhibits its domain as the stabilization of its target.

Now suppose that $\V$ is an additive $\infty$-category. Then the functor $\fib$ is an equivalence, with (left adjoint) inverse sending $(X, Y)$ to $X\lrar X\oplus Y\lrar X$ (see e.g.\ \cite[Lemma 1.5.12]{9author}). We consequently obtain in this case an equivalence 
$$
\xymatrix{
\T\V\ar[rr]^-{\simeq}\ar[dr]_{\pi} && \V\times\Sp(\V) \ar[dl]^-{\pi_1} \\
& \V & \\
}
$$
between $\T\V$ and the fiberwise stabilization of $\V\times \V$ over $\V$. If $\V$ is furthermore \emph{stable} then both $G$ and $\fib$ are equivalences and one obtains an identification $\T\V\simeq \V\times \V$ such that $\pi(X, Y)\simeq X$ and $\Omega^\infty(X, Y)\simeq X\oplus Y$. It then follows that for any additive presentable $\infty$-category $\V$, the functor $\Sigma^\infty\colon \V\lrar \Sp(\V)$ to its stabilization fits in a pullback square of tangent categories
$$\xymatrix{
\T\V\ar[r]\ar[d] & \T(\Sp(\V))\ar[d]\\
\V\ar[r] & \Sp(\V).
}$$
\end{example}

\begin{defn}\label{d:aff}
Let $\V$ be a presentable $\infty$-category. Then the inclusion $\T\V\lrar \Fun(\cS^{\fin}_\ast, \V)$ admits a left adjoint, which we will denote by $X\mapsto X^{\exc}$. We will say that a map $X\lrar Y$ in $\Fun(\cS^{\fin}_\ast, \V)$ is a \textbf{$\T\V$-local equivalence} if the map $X^{\exc}\lrar Y^{\exc}$ is an equivalence.
\end{defn}
\begin{rem}\label{r:generating local equiv}
For any $S\in \cS^{\fin}_\ast$ and $C\in \V$, let $h_S\otimes C=\Map(S, -)\otimes C$ be the corresponding corepresentable functor, i.e.\ the left Kan extension of $C\colon \ast\lrar \V$ along $S\colon \ast\lrar \cS^{\fin}_\ast$. Note that $F\in \Fun(\cS^{\fin}_\ast, \V)$ is excisive if and only if it is a local object with respect to the set of maps
\begin{equation}\label{d:generating local equiv}
\big(h_{S_1}\coprod_{h_{S_3}} h_{S_2}\big)\otimes C_\alpha \lrar \big(h_{S_0}\otimes C_\alpha)
\end{equation}
for any set of generators $\{C_\alpha\}$ of $\V$ and any pushout square in $\cS^{\fin}_\ast$
$$\xymatrix{
S_0\ar[r]\ar[d] & S_1\ar[d]\\
S_2\ar[r] & S_3.
}$$
In particular, the $\T\V$-local equivalences are strongly generated by this set of maps \cite[Proposition 5.5.4.15]{Lur09}.
\end{rem}
\begin{rem}\label{r:excisive approximation preserves cartesian arrows}
For any presentable $\infty$-category $\V$, the description of the generating $\T\V$-local equivalences from Remark \ref{r:generating local equiv} shows that evaluation at $\ast\in \cS^{\fin}_\ast$ sends $\T\V$-local equivalences to equivalences in $\V$. It follows that there is a commuting diagram
$$\xymatrix@C=3pc{
\T\V\ar@{^{(}->}[0, 0]+<3ex, 0ex>;[r]\ar[rd] & \Fun(\cS_{\ast}^{\fin}, \V)\ar[r]^-{(-)^{\exc}} \ar[d] & \T\V\ar[ld]\\
& \V.
}$$
The vertical functors are Cartesian (and coCartesian) fibrations, with right adjoint sections taking the constant $\cS^{\fin}_\ast$-diagram. In particular, an arrow in $\T\V$ or $\Fun(\cS^{\fin}_\ast, \V)$ is Cartesian if and only if it is the pullback of a map between constant diagrams. It follows that the (right adjoint) inclusion $\T\V\hookrightarrow \Fun(\cS^{\fin}_\ast, \V)$ preserves Cartesian arrows. When $\V$ is compactly generated, or more generally differentiable \cite[Definition 6.1.1.6]{Lur14} (cf.\ Remark \ref{r:differentiable}), the functor $(-)^{\exc}$ preserves Cartesian arrows by \cite[Theorem 6.1.1.10]{Lur14}.
\end{rem}
Recall that every $E\in \T_X\V$ defines a reduced excisive functor $E\colon \cS^{\fin}_\ast\lrar \V_{/X}$. Restricting $E$ to the full subcategory of finite pointed sets, we obtain a very special $\Gamma$-space object in $\V_{/X}$ in the sense of Segal, whose underlying object is $\Omega^\infty(E)$. In other words, $\Omega^\infty(E)$ has the structure of a grouplike $\EE_\infty$-monoid in the sense of \cite{GGN15}. 

Conversely, in the presence of loop space machinery, every grouplike $\EE_\infty$-monoid arises from a spectrum. For later purposes, let us make this slightly more precise: suppose that $\V$ is a presentable $\infty$-category and let 
$$
\Grp(\V)\subseteq \Fun(\Delta^{\op}, \V) \qquad\qquad \Grp_{\EE_\infty}(\V)\subseteq \Fun(\Fin_\ast, \V)
$$
denote the $\infty$-categories of grouplike monoids, resp.\ grouplike $\EE_\infty$-monoids in $\V$. Note that both arise as full (reflective) subcategories of diagrams satisfying the grouplike Segal conditions \cite[\S 2.4.2]{Lur14}, \cite{GGN15}. In addition, there is an adjoint pair
\begin{equation}\label{e:delooping}\vcenter{\xymatrix{
B\colon \Grp(\V)\ar@<1ex>[r] & \V_\ast\ar@<1ex>[l]_-\perp\colon \Omega
}}\end{equation}
where the left adjoint sends a grouplike monoid to its bar construction and the right adjoint sends a pointed object in $\V$ to its loop space (endowed with the group structure coming from the usual cogroup structure $S^1\lrar S^1\vee S^1$).
\begin{defn}
Let $\V$ be a presentable $\infty$-category. We will say that $\V$ \textbf{has loop space machinery} if it satisfies the following conditions:
\begin{enumerate}
\item the Cartesian product $\V\times\V\stackrel{\times}{\lrar} \V$ preserves geometric realizations.
\item the unit of the adjunction \eqref{e:delooping} is an equivalence.
\end{enumerate}
\end{defn}
\begin{example}\label{ex:loopspace}
All $\infty$-toposes and stable presentable $\infty$-categories have loop space machinery. More generally, a pre-stable presentable $\infty$-category (i.e.\ the connective part of a $t$-structure on a stable $\infty$-category \cite[\S C.1]{Lur16}) has loop space machinery. If $\V$ has loop space machinery and $U\colon \W\lrar \V$ is a right adjoint functor preserving sifted colimits and detecting equivalences (in particular, it is monadic), then $\W$ has loop space machinery.
\end{example}
\begin{pro}\label{p:gamma objects as spectra}
Let $\V$ be a presentable $\infty$-category with loop space machinery and consider the adjoint pair
$$\xymatrix{
\mathbf{B}\colon \Grp_{\EE_\infty}(\V)\ar@<1ex>[r] & \Sp(\V)=\Exc_{\red}(\cS_\ast^{\fin}, \V)\colon \Omega^\infty\ar@<1ex>[l]_-\perp
}$$
whose right adjoint restricts a reduced excisive functor along the inclusion $i\colon \Fin_\ast\lrar \cS^{\fin}_\ast$. Then the left adjoint $\mathbf{B}$ is fully faithful and a functor $F\colon \cS_\ast^{\fin}\lrar \V$ lies in its essential image (in particular, it will be reduced excisive) if and only if it satisfies the following two conditions:
\begin{enumerate}
\item\label{it:segal} Its restriction to $\Fin_\ast$ satisfies the grouplike Segal conditions.
\item\label{it:preserve sifted} It preserves all \textbf{finite geometric realizations}, i.e.\ colimits of simplicial diagrams in $\cS_\ast^{\fin}$ that are left Kan extended from $\Delta_{\leq n}^{\op}\subseteq \Delta^{\op}$, for some $n$.
\end{enumerate}
\end{pro}
\begin{lem}\label{l:left kan extension sifted colim}
Let $i\colon \Fin_*\lrar \cS^{\fin}_\ast$ be the natural fully faithful inclusion. Then restriction and left Kan extension define an adjoint pair
$$\xymatrix{
i_!\colon \Fun(\Fin_\ast, \V)\ar@{^{(}->}@<1ex>[r] & \Fun(\cS_\ast^{\fin}, \V)\colon i^*\ar@<1ex>[l]_-\perp
}$$
whose left adjoint is fully faithful. The essential image of $i_!$ consists exactly of those functors $F\colon \cS_\ast^{\fin}\lrar \V$ that preserve finite geometric realizations.
\end{lem}
\begin{proof}
Note that $i_!$ is fully faithful because $i$ is. To identify the essential image, let us factor the Yoneda embedding as
$$\xymatrix{
\Fin_*\ar[r]^-{i} & \cS^{\fin}_\ast \ar[r]^-{j} & \mathrm{PSh}(\Fin_*)
}$$
where $j$ sends $T\in \cS^{\fin}_\ast$ to $\Map_{\cS^{\fin}_\ast}(i(-), T)$. Note that for each finite pointed set $S$, $\Map_{\cS^{\fin}_\ast}(i(S), -)$ preserves all finite geometric realizations in $\cS_\ast^{\fin}$ (since it sends $T\mapsto T^{\times |S|-1}$). Consequently, the functor $j$ preserves finite geometric realizations as well. Because every finite pointed set is the geometric realization of some $n$-skeletal simplicial diagram in $\Fin_*$ (and the Yoneda embedding is fully faithful on $\Fin_*$), it follows that $j$ is fully faithful.

We then have a sequence of adjunctions given by restriction and left Kan extension
$$\xymatrix{
\Fun(\Fin_\ast, \V)\ar@{^{(}->}@<1ex>[r]^-{i_!} & \Fun(\cS_\ast^{\fin}, \V)\ar@<1ex>[l]_-\perp^-{i^*} \ar@{^{(}->}@<1ex>[r]^-{j_!} & \Fun\big(\mathrm{PSh}(\Fin_\ast), \V\big)\ar@<1ex>[l]_-\perp^-{j^*}
}$$
where the left adjoints are fully faithful. By \cite[Lemma 5.1.5.5]{Lur09}, the essential image of $j_!i_!$ coincides with those functors $\mathrm{PSh}(\Fin_\ast)\lrar \V$ preserving all colimits. Consequently, the essential image of $i_!$ consists of those functors whose left Kan extension along $j$ defines a colimit-preserving functor $\mathrm{PSh}(\Fin_*)\lrar \V$.

Since $j$ preserves finite geometric realizations, it follows that any functor in the image of $i_!$ preserves finite geometric realizations. Conversely, given $F\colon \cS^{\fin}_\ast\lrar \V$ preserving finite geometric realizations, we have to verify that the counit map 
$$
i_!i^*F(T)\lrar F(T)
$$
is a natural equivalence for $T\in \cS^{\fin}_\ast$. Note that the domain and codomain both preserve finite geometric realizations in $T$. Since each $T$ is the realization of a finite simplicial diagram in $\Fin_*$, we can reduce to the case where $T\in \Fin_*$. But $F$ and $i_!i^*F$ agree on finite pointed sets by construction. 
\end{proof}
\begin{proof}[Proof of Proposition \ref{p:gamma objects as spectra}]
Consider the adjoint pair $(i_!, i^*)$ from Lemma \ref{l:left kan extension sifted colim}. It suffices to verify that for every $A\colon \Fin_\ast\lrar \V$ satisfying the grouplike Segal conditions, the functor $F:=i_!(A)\colon \cS_\ast^{\fin}\lrar\V$ is reduced excisive. Indeed, in that case the adjoint pair $(i_!, i^*)$ simply restricts to spectra and grouplike $\EE_\infty$-monoids, i.e.\ $\mathbf{B}=i_!$ and $\Omega^\infty=i^*$.

To see this, let us first observe that \cite[Proposition 2.4.3.9]{Lur14} (applied to the functor $F^{\op}$) shows that $F\colon \cS^{\fin}_\ast\lrar \V$ has a natural oplax symmetric monoidal structure with respect to $\big(\cS^{\fin}_\ast, \vee\big)$ and $\big(\V, \times\big)$. More specifically, the structure maps of this oplax symmetric monoidal structure are given by the maps 
\begin{equation}\label{d:oplax monoidal map}
F(T_1\vee T_2)\lrar F(T_1)\times F(T_2)
\end{equation}
for $T_1, T_2\in \cS^{\fin}_\ast$, whose first component is induced by the map $T_1\vee T_2\lrar T_1$ collapsing $T_2$ and the second component by the map $T_1\vee T_2\lrar T_2$ collapsing $T_1$. The Segal conditions for $A$ imply that \eqref{d:oplax monoidal map} is an equivalence when $T_1, T_2$ are contained in $\Fin_*$. Since both the domain and codomain of \eqref{d:oplax monoidal map} preserve finite geometric realizations in $T_1$ and $T_2$, it follows that \eqref{d:oplax monoidal map} is a natural equivalence.  In other words, $F$ has a natural symmetric monoidal structure. This implies that $F$ fits into a commuting diagram of functors
$$\xymatrix{
\Alg_{\mathbb{E}_\infty}(\cS^{\fin}_\ast, \vee)\ar[r]^-{\tilde{F}}\ar[d]_{\mathrm{forget}}^\sim & \Mon_{\mathbb{E}_\infty}(\V)\ar[d]^{\mathrm{forget}}\\
\cS^{\fin}_\ast\ar[r]_-{F} & \V.
}$$
Here the left vertical functor is an equivalence by  \cite[Proposition 2.4.3.9]{Lur14}. In fact, the functor $\tilde{F}$ takes values in the full sub-$\infty$-category $\Grp_{\mathbb{E}_\infty}(\V)\subseteq \Mon_{\mathbb{E}_\infty}(\V)$ of grouplike $\mathbb{E}_\infty$-monoids in $\V$. Indeed, since $\tilde{F}$ preserves finite geometric realizations (which are computed on the underlying object) and $\Grp_{\mathbb{E}_\infty}(\V)\subseteq \Mon_{\mathbb{E}_\infty}(\V)$ is closed under finite geometric realizations, it suffices to verify that $\tilde{F}(S)\in \Grp_{\mathbb{E}_\infty}(\V)$ for $S\in \Fin_*$. In this case, this follows from the (grouplike) Segal conditions.

To see that $F$ is reduced excisive, it now suffices to verify that $\tilde{F}$ is reduced excisive (since the forgetful functor $\mathrm{Grp}_{\mathbb{E}_\infty}(\V)\lrar \V$ preserves all finite limits). Clearly $\tilde{F}$ is reduced. To see that it is excisive, it suffices to verify that for every $S\in \cS^{\fin}_\ast$, the natural map
\begin{equation}\label{e:excision via suspension}\xymatrix{
\tilde{F}(S)\ar[r] & \Omega \tilde{F}(\Sigma S)
}\end{equation}
is an equivalence \cite[Proposition 1.4.2.13]{Lur14}. 
Using the simplicial model for the $S^1$ given by $\Delta^1/\partial\Delta^1$ and the fact that $\tilde{F}$ preserves finite sifted colimits, one can compute $\tilde{F}(\Sigma S)$ as the geometric realization
$$\xymatrix{
\tilde{F}\big(\Sigma S\big) & \ar[l] \tilde{F}(\ast) & \ar@<0.5ex>[l]\ar@<-0.5ex>[l] \tilde{F}(S) & \tilde{F}(S\vee S)\ar@<-1ex>[l]\ar[l]\ar@<1ex>[l] & \dots \ar@<0.5ex>[l]\ar@<-0.5ex>[l]\ar@<1.5ex>[l]\ar@<-1.5ex>[l]
}$$
This is precisely the bar construction of (the underlying grouplike monoid of) $\tilde{F}(S)$. The map \eqref{e:excision via suspension} can therefore be identified with the map $\tilde{F}(S)\lrar \Omega B(\tilde{F}(S))$, which is an equivalence since $\V$ has loop space machinery.
\end{proof}

\subsection{Monoidal structure on the tangent bundle}
Our next goal will be to construct a (closed) symmetric monoidal structure on the tangent bundle $\T\V$ of a closed symmetric monoidal presentable $\infty$-category. To this end, recall recall that if $\V$ is a symmetric monoidal $\infty$-category and $\I$ is an $\infty$-category, then the objectwise tensor product endows $\Fun(\I, \V)$ with the natural structure of a symmetric monoidal $\infty$-category. For every $f\colon \I\lrar \J$, the restriction functor $f^*\colon \Fun(\J, \V)\lrar \Fun(\I, \V)$ is a symmetric monoidal functor. 

\begin{pro}\label{p:tens}
Let $\V$ be a presentable SM $\infty$-category and endow $\Fun(\cS^{\fin}_\ast, \V)$ with the levelwise tensor product $\otimes_{\lev}$. Then the localization of $\Fun(\cS^{\fin}_\ast, \C)$ at the $\T\V$-local equivalences is monoidal. In particular, there exists a closed symmetric monoidal structure on $\T\V$ given by
$$
X\otimes Y = \big(X\otimes_{\lev} Y\big)^{\exc}.
$$
\end{pro}
\begin{proof}
By \cite[Proposition 4.1.7.4]{Lur14}, it suffices to verify that $X\otimes Y\lrar X\otimes Y'$ is a $\T\V$-local equivalence for every $X: \cS^{\fin}_\ast\lrar \V$ and every $\T\V$-local equivalence $Y\lrar Y'$. Since the $\T\V$-local equivalences are closed under colimits and $\otimes_\text{lev}$ preserves colimits in each variable, we may assume that $Y\lrar Y'$ is a generating local equivalence of the form \eqref{d:generating local equiv} and that $X=h_T\otimes D$. Now note that $\otimes_\text{lev}$ is the monoidal structure obtained by applying Day convolution \cite[Example 2.2.6.17]{Lur14} to $(\cS^{\fin}_\ast, \vee)$ and $(\V, \otimes)$ (cf.\ \cite[Theorem 2.4.3.18]{Lur14}). Consequently, it suffices to show that the map
$$
\big(h_{T\vee S_1}\coprod_{h_{T\vee S_3}} h_{T\vee S_2}\big)\otimes (D\otimes C) \lrar h_{T\vee S_0}\otimes (D\otimes C)
$$
is a $\T\V$-local equivalence. This is obvious since
$$\xymatrix{
T\vee S_0\ar[r]\ar[d] & T\vee S_1\ar[d]\\
T\vee S_2\ar[r] & T\vee S_3
}$$
is a pushout square in $\cS^{\fin}_\ast$.
\end{proof}
\begin{rem}\label{r:tensoring}
Evaluation at $\ast$ defines a colimit-preserving functor $\Fun(\cS^{\fin}_\ast,\V)\lrar \V$ which is symmetric monoidal for the levelwise tensor product and sends $\T\V$-local equivalences to equivalences in $\V$. It follows that $\pi\colon \T\V\lrar \V$ is symmetric monoidal for $\otimes$ as well. Its left adjoint is also symmetric monoidal, so that $\T\V$ is tensored over $\V$ via the formula
$$
C\otimes X = (C\otimes_\text{lev} X(-))^{\text{aff}}.
$$
Furthermore, note that the functor $\Om^\infty: \T\V\lrar \V$ is lax symmetric monoidal, being the composite of the lax symmetric monoidal inclusion $\T\V\lrar \Fun(\cS^{\fin}_\ast, \V)$ and the symmetric monoidal functor $\ev_{S^0}\colon \Fun(\cS^{\fin}_\ast, \V)\lrar \V$ (for the levelwise tensor product on the domain). On the other hand, $\Sigma^\infty_+$ is not lax monoidal.
\end{rem}
\begin{rem}\label{r:via PrL}
Let $\C$ be a SM $\infty$-category and let $\textbf{1}\lrar \C$ be the symmetric monoidal functor given by the inclusion of the unit object. By functoriality of the Day convolution product, for any presentable SM $\infty$-category $\V$, this map and the unit map $\cS\lrar \V$ in $\CAlg(\Pr^\mathrm{L})$ induce a commuting square in $\CAlg(\Pr^\mathrm{L})$
$$\xymatrix{
\cS\ar[r]\ar[d] & \V\ar[d]\\
\big(\Fun(\C, \cS), \otimes_{\mathrm{Day}}\big)\ar[r] & \big(\Fun(\C, \V), \otimes_{\mathrm{Day}}\big).
}$$
It follows that there is a symmetric monoidal functor $\V\otimes \Fun(\C, \cS)\lrar \Fun(\C, \V)$, whose underlying functor coincides with the equivalence from \cite[Proposition 4.8.1.17]{Lur14}. 

Now let $\T\cS\in \CAlg(\Pr^\mathrm{L})$ be the tangent bundle to spaces, endowed with the SM-structure from Proposition \ref{p:tens}. Since the monoidal structure on $\T\cS$ arises as a monoidal localization of $\Fun(\cS^{\fin}_\ast, \cS)$ with the Day convolution product, applying the above observation to $\C=\cS^{\fin}_\ast$, it follows that for any other $\V\in \CAlg(\Pr^\mathrm{L})$, one can identify $\T\V\simeq \T\cS\otimes \V$ in $\CAlg(\Pr^\mathrm{L})$ (cf.\ \cite[Remark 4.8.1.13]{Lur14}).
\end{rem}
Note that for any functor $X\colon \cS^{\fin}_\ast\lrar \V$, there is a canonical (counit) map $X(\ast)\lrar X$, where we consider $X(\ast)\in\V$ as a constant diagram.
\begin{lem}\label{lem:pushout product}
Let $X, Y\colon \cS^{\fin}_\ast\lrar \V$ be two functors. Then the pushout-product map
$$\xymatrix{
\psi(X, Y)\colon X(\ast)\otimes_{\lev} Y\coprod_{X(\ast)\otimes_{\lev} Y(\ast)} X\otimes_{\lev}Y(\ast)\ar[r] & X\otimes_{\lev} Y
}$$
is a $\T\V$-local equivalence.
\end{lem}
\begin{proof}
Suppose that $X=\colim X_i$ for some diagram of functors $X_i$. Since evaluation and taking the constant diagram preserve colimits, we can identify the pushout-product map $\psi(X, Y)$ with the colimit $\colim_i \psi(X_i, Y)$ in the arrow category of $\Fun(\cS^{\fin}_\ast, \V)$. As $\T\V$-local equivalences are stable under colimits, we can therefore reduce to the case where $X=h_S\otimes C$ and $Y=h_T\otimes D$ are corepresentables.

Using that the constant diagram on $X(\ast)$ is given by $h_\ast\otimes X(\ast)$, the pushout-product map can then be identified with
$$\xymatrix{
h_T\otimes (C\otimes D)\coprod\limits_{h_\ast\otimes (C\otimes D)}h_S\otimes(C\otimes D)\ar[r] & (h_{S}\otimes C)\otimes_\text{lev} (h_{T}\otimes D).
}$$
Since the levelwise tensor product agrees with the Day convolution product (cf.\ \cite[Theorem 2.4.3.18]{Lur14}), the codomain can be identified with $h_{S\vee T}\otimes(C\otimes D)$. The above map is then a $\T\V$-local equivalence because
$$\xymatrix{
S\vee T\ar[r]\ar[d] & S\vee \ast\ar[d]\\
\ast\vee T\ar[r] & \ast\vee \ast
}$$
is a coCartesian square (see Remark \ref{r:generating local equiv}).
\end{proof}
The above lemma can be described somewhat informally as follows: we can identify an object of $\T\V$ with a tuple consisting of $C\in \V$ and $E\in \Sp(\V_{/C})$. Using the tensoring of $\T\V$ over $\V$ from Remark \ref{r:tensoring}, we then have an equivalence
\begin{equation}\label{e:square zero tensor}
(C, E)\otimes (D, F)\simeq \big(C\otimes D, (C\otimes F)\oplus (E\otimes D)\big)
\end{equation}
where the direct sum is taken in the fiber $\T_{C\otimes D}\V$. This justifies the following terminology:
\begin{define}\label{d:square zero tensor product}
Let $\V$ be a presentable SM $\infty$-category. The \textbf{square zero tensor product} on $\T\V$ is the symmetric monoidal structure provided by Proposition \ref{p:tens}.
\end{define}
\begin{pro}\label{p:tangent to algebras}
Let $\V$ be a presentable SM $\infty$-category and let $\O$ be an $\infty$-operad. Then there is an equivalence of $\infty$-categories
$$
\Alg_{\O}(\T\V)\simeq \T(\Alg_\O(\V))
$$
where $\T\V$ is endowed with the square zero monoidal structure.
\end{pro}
\begin{proof}
It follows from the commutativity of the Boardman--Vogt tensor product \cite[Proposition 2.2.5.13]{Lur14} that there is an equivalence of symmetric monoidal $\infty$-categories
$$
\Alg_\O\big(\Fun(\cS^{\fin}_\ast, \V), \otimes_{\lev}\big) \simeq  \Fun(\cS^{\fin}_\ast, \Alg_{\O}(\V)).
$$
	This equivalence identifies excisive diagrams of $\O$-algebras in $\V$ with levelwise $\O$-algebras in $\Fun(\cS^{\fin}_\ast, \V)$ whose underlying functor is excisive. The fully faithful functor $\T\V\lrar \Fun(\cS^{\fin}_\ast, \V)$ is lax symmetric monoidal and hence realizes $\T\V^{\otimes}$ as a full suboperad of the $\infty$-operad $\Fun(\cS^{\fin}_\ast, \V)^{\otimes}$. It follows that $\O$-algebras in $\T\V$ are equivalent to $\O$-algebras in $\Fun(\cS^{\fin}_\ast, \V)$ whose underlying functor is excisive. In turn, these are equivalent to functors $\cS^{\fin}_\ast\lrar \Alg_{\O}(\V)$ that are excisive, because the forgetful functor from $\O$-algebras to $\V$ detects limits \cite[Corollary 3.2.2.4]{Lur14}.
\end{proof}
The square zero monoidal structure can be made more explicit in the case where $\V$ is a stable presentable monoidal $\infty$-category, using the following:
\begin{define}
Let $\D$ be a presentable SM $\infty$-category. We will say that an object $D\in \D$ is \textbf{square zero} if the canonical map $\emptyset\lrar D\otimes D$ is an equivalence and denote by $\SqZ(\D)\subseteq \D$ the full subcategory on the square zero objects. Note that every symmetric monoidal left adjoint $F\colon \D\lrar \D'$ restricts to a natural map $F\colon \SqZ(\D)\lrar \SqZ(\D')$.
\end{define}
Recall that the $\infty$-category of \textbf{$\V$-linear SM $\infty$-categories} is given by the $\infty$-category $\CAlg_\V(\Pr^{\mathrm{L}})\simeq \CAlg(\Pr^{\mathrm{L}})_{\V/}$ of presentable SM $\infty$-categories $\D$ equipped with a symmetric monoidal left adjoint functor $\V\lrar \D$.
\begin{define}
Let $\W$ be a $\V$-linear SM $\infty$-category together with a square zero object $M\in \W$. We say that this exhibits $\W$ as the \textbf{free $\V$-algebra on a square zero object} if evaluation at $M$ defines a natural equivalence
$$
\ev_{M}\colon \Fun^{\otimes}_\V(\W, \D)\lrar\SqZ(\D).
$$
\end{define}
\begin{rem}\label{r:universal preserved by tensor}
Consider a pushout square in $\CAlg(\Pr^\mathrm{L})$
$$\xymatrix{
\V_1\ar[r]\ar[d] & \V_2\ar[d]\\
\W_1\ar[r]_{f} & \W_2.
}$$ 
If $M\in \W_1$ exhibits $\W_1$ as the free $\V_1$-algebra on a square zero object, then $f(M)$ exhibits $\W_2\simeq \V_2\otimes_{\V_1}\W_1$ as a free $\V_2$-algebra on a square zero object: indeed, the evaluation at $f(M)$ factors as two equivalences
$$\xymatrix{
\ev_{f(M)}\colon \Fun^{\otimes}_{\V_2}(\V_2\otimes_{\V_1} \W_1, \D)\ar[r]_-\sim^-{f^*} & \Fun^{\otimes}_{\V_1}(\W_1, \D)\ar[r]_-\sim^-{\ev_M} & \SqZ(\D).
}$$
\end{rem}
\begin{lem}\label{lem:free square zero algebra}
For every presentable SM $\infty$-category $\V$, there exists a free $\V$-algebra $\V[\epsilon]$ on a square zero object.
\end{lem}
\begin{proof}
By Remark \ref{r:universal preserved by tensor}, it suffices to check the case $\V=\cS$. Now consider the $\infty$-category of symmetric sequences $\Fin^{\mathrm{bij}}\lrar \cS$, equipped with the Day convolution product with respect to the disjoint union on $\Fin^{\mathrm{bij}}$. Then the monoidal unit is the corepresentable $h_0$, while evaluation at $h_1$ exhibits the $\infty$-category of symmetric sequences as the free presentable SM $\infty$-category on an object \cite[Proposition 2.2.4.9, Corollary 4.8.1.12]{Lur14}. 

Let us now denote by $\cS[\epsilon]$ the (reflective) localization of symmetric sequences at the set of maps $\emptyset\lrar h_n$ for all $n\geq 2$. Since the pushout-product map of $\emptyset\lrar h_n$ and $\emptyset\lrar h_m$ is equivalent to the map $\emptyset\lrar h_{n+m}$, this exhibits $\cS[\epsilon]$ as a symmetric monoidal localzation of $\Fun(\Fin^{\mathrm{bij}}, \cS)$. By the universal property of symmetric monoidal localizations, the square zero object $h_1\in \cS[\epsilon]$ then realizes $\cS[\epsilon]$ as the free presentable SM $\infty$-category on a square zero object.
\end{proof}
\begin{rem}\label{r:free square zero algebra}
The proof of Lemma \ref{lem:free square zero algebra} shows that $\cS[\epsilon]$ is equivalent to the $\infty$-category of symmetric sequences which are contractible in arities $\geq 2$. This $\infty$-category can be identified with the free presentable $\infty$-category generated by two objects $A$ (arity $0$) and $M$ (arity $1$). Furthermore, its symmetric monoidal structure is determined by $A\otimes A=A, M\otimes A=M$ and $M\otimes M\simeq \emptyset$.
\end{rem}
\begin{cons}
Let $\V$ be a stable presentable SM $\infty$-category and consider the following cofiber sequence of excisive functors
$$\xymatrix{
h_{\ast}\otimes 1_{\V}\ar[r]^-{i\otimes \id} & h_{S^{0}}\otimes 1_{\V}\ar[r] & M_\V
}$$ 
where $i\otimes \id$ is the canonical map of corepresentables induced by $\ast\lrar S^0$. Lemma \ref{lem:pushout product} shows that the pushout-product of $h_{\ast}\otimes 1_\V\lrar h_{S^0}\otimes 1_\V$ with itself in $\Fun(\cS^{\fin}_\ast, \V)$ is a $\T\V$-local equivalence. Consequently, the pushout-product map in the monoidal localization $\T\V$ becomes an equivalence. Since the cofiber of a pushout-product map is the tensor product of the cofibers, it follows that $M_{\V}\otimes M_{\V}\simeq 0$ in $\T\V$.
\end{cons}
\begin{pro}\label{p:squarezero for stable}
Let $\V$ be a stable presentable SM $\infty$-category and consider $\T\V$ as a $\V$-linear SM $\infty$-category via the left adjoint $\V\lrar \T\V$ to the projection. Then the square zero object $M_\V\in \T\V$ exhibits $\T\V$ is the free $\V$-algebra on a square zero object.
\end{pro}
\begin{proof}
Let us first consider the case where $\V=\Sp$. Let $\Sp[\epsilon]\simeq \Sp\otimes \cS[\epsilon]$ be the free $\Sp$-algebra on a square zero object $M$. By Remark \ref{r:free square zero algebra}, $\Sp[\epsilon]$ can be identified with the free stable presentable $\infty$-category on two objects $A, M$. Furthermore, the monoidal structure on $\Sp[\epsilon]$ has unit $A$ and $M\otimes M\simeq \emptyset$.
By the universal property, we now obtain a (symmetric monoidal) left adjoint functor
$$
\phi\colon \Sp[\epsilon]\simeq \Fun(\{A, M\}, \Sp)\lrar \T\Sp
$$
which is determined on generators by $\phi(h_A)=1_{\T\Sp}=h_\ast\otimes 1_{\Sp}$ and $\phi(h_M)=M_{\Sp}=\cof\big(h_\ast\otimes 1_{\Sp}\lrar h_{S^0}\otimes 1_{\Sp}\big)$. It follows that the right adjoint to $\phi$ is is given by the composite functor $\T\Sp \lrar \Fun(\mathrm{Ret}, \Sp) \lrar \Sp \times \Sp$ appearing in Example \ref{e:stable}, which is an equivalence since $\Sp$ is stable.

A general stable presentable SM $\infty$-category $\V$ comes equipped with a canonical symmetric monoidal functor $f\colon \Sp\lrar \V$. Sine the induced functor $\T\Sp\lrar \T\V$ sends $M_{\Sp}$ to $M_{\V}$, the assertion now follows by tensoring the equivalence $\Sp[\epsilon]\lrar \T\Sp$ over $\Sp$ with $\V$ (see Remark \ref{r:via PrL} and \ref{r:universal preserved by tensor}).
\end{proof}
\begin{rem}\label{r:square zero in additive setting}
Let $\V$ be an additive presentable SM $\infty$-category. By \cite[Theorem 5.1]{GGN15}, its stabilization $\Sp(\V)$ carries an induced symmetric monoidal structure and $\Sigma^\infty\colon \V\lrar \Sp(\V)$ is a symmetric monoidal functor. It then follows from Example \ref{e:stable} that there is a pullback square of presentable SM $\infty$-categories
$$\xymatrix{
\T\V\ar[r]\ar[d] & \T(\Sp(\V))\ar[d]\\
\V\ar[r]^-{\Sigma^{\infty}} & \Sp(V).
}$$
In particular, identifying $\T\V\simeq \V\times \Sp(\V)$, Proposition \ref{p:squarezero for stable} then gives an explicit description of the square zero monoidal structure, given informally by the formula
\begin{equation}\label{eq:square zero tensor product additive case}
(C, E)\otimes_{\T\V} (D, F)\simeq \big(C\otimes_{\V} D, (\Sigma^{\infty}C\otimes_{\Sp(\V)} F)\oplus (E\otimes_{\Sp(\V)}\Sigma^{\infty}D)\big).
\end{equation}
\end{rem}
\begin{rem}\label{r:square zero via day convolution}
Let $\V$ be a presentable SM $\infty$-category. Then the free $\V$-algebra $\V[\epsilon]$ on a square zero object can also be described in terms of a variant of the Day convolution product applicable to \emph{promonoidal} $\infty$-categories, as developed in recent work of Nardin-Shah \cite{NS20}. More precisely, one can check that the 2-coloured operad $\MCom$ for commutative algebras and modules is such a promonoidal ($\infty$-)category. Since the underlying category of $\MCom$ is simply the set $\{A, M\}$, this endows $\Fun(\{A, M\}, \V)$ with a Day convolution product which has the property that 
\begin{align*}
(h_A\otimes C)\otimes (h_A\otimes D)&=h_A\otimes (C\otimes D)\\
(h_A\otimes C)\otimes (h_M\otimes D)&=h_M\otimes (C\otimes D)\\
(h_M\otimes C)\otimes (h_M\otimes D) &=\emptyset.
\end{align*}
In particular, the square zero object $h_M\otimes 1_{\V}$ induces a symmetric monoidal functor from $\V[\epsilon]$ to this Day convolution, which is easily seen to be an equivalence. The universal property of the Day convolution therefore implies that for any $\infty$-operad $\O$, there is a natural equivalence
$$
\Alg_{\O}\big(\V[\epsilon]\big)\simeq \Alg_{\O\times\MCom}(\V).
$$
The $\infty$-operad $\M\O=\O\times \MCom$ is the $\infty$-operad for $\O$-algebras and (operadic) modules over them \cite{part1, hin-rectification}. Combining this with Proposition \ref{p:tangent to algebras} and \ref{p:squarezero for stable}, one finds that $\T\Alg_{\O}(\V)\simeq \Alg_{\M\O}(\V)$ for every stable presentable SM $\infty$-category $\V$ (see also \cite{Sch97, BM, Lur14}).
\end{rem}

\subsection{Square zero extensions}\label{s:square zero}
Let $\V$ be an $\infty$-category with finite limits and $B\in\V$ an object. For every $E\in \T_B\V$, there is a natural map $\Omega^{\infty}(E)\lrar B$, arising from the map of finite pointed spaces $S^0\lrar \ast$. For every map $X\lrar B$, we denote by
$$
H^0_\Q(X; E)=\pi_0\Map_{/B}(X, \Omega^\infty(E))
$$ 
the set of homotopy classes of lifts $\eta\colon X\lrar \Omega^\infty(E)$. Since $\Omega^{\infty}(E)$ is a grouplike $\EE_\infty$-monoid over $B$, this forms an abelian group; its unit is the zero map $0\colon X\lrar B\lrar \Omega^\infty(E)$ induced by the map of finite pointed spaces $\ast\lrar S^0$. More generally, we will refer to the groups $H^n_{\Q}(X; E)=H^0_{\Q}(X; \Sigma^nE)$ as the $n$-th \textbf{Quillen cohomology} groups of $X$ with coefficients in $E$. 
Given a section $\eta\colon X\lrar \Omega^{\infty}(E)$, we will say that the pullback square
$$\xymatrix{
\tilde{X}\ar[r]\ar[d] & B\ar[d]^0\\
X\ar[r]_-\eta & \Omega^\infty(E)
}$$
exhibits $\tilde{X}$ as a \textbf{square zero extension} of $X$ \cite[Definition 7.4.1.6]{Lur14}.
\begin{rem}
The above definition of a square zero extension looks slightly more general than the one appearing in \cite[Definition 7.4.1.6]{Lur14}, where it is assumed that $B=X$. However, note that there is a natural map $p\colon X\lrar B$ (induced by the projection $\Omega^{\infty(E)}\lrar B$); pulling back the parametrized spectrum $E$ along $p$, one can also realize $\tilde{X}$ as the square zero extension of $X$ classified by the canonical map $\eta'\colon X\lrar \Omega^\infty(p^*E)$.
\end{rem}
\begin{example}\label{e:square zero extensions in stable}
Let $\V$ be a stable presentable SM $\infty$-category and $\O$ an $\infty$-operad. For any $A\in \Alg_{\O}(\V)$, $\T_A\Alg_{\O}(\V)$ can be identified with the $\infty$-category of operadic $A$-modules (see \cite[Corollary 1.0.5]{part1} or \cite[Theorem 7.3.4.13]{Lur14}). Alternatively, Remark \ref{r:square zero via day convolution} identifies $\T\Alg_{\O}(\V)\simeq \Alg_{\M\O}(\V)$ with the $\infty$-category of $\O$-algebras and modules over them.

Now, given such an $A$-module $E$, the $\O$-algebra $\Omega^\infty(E)$ can be identified with the split square zero extension $A\oplus E$. For any section $\eta\colon A\lrar A\oplus E$, the resulting pullback $A_\eta\lrar A$ then indeed behaves like a square zero extension: its fiber is given by the desuspension $E[-1]$, together with the trivial $\O$-algebra structure (see e.g.\ \cite[Proposition 7.4.1.14]{Lur14}).
\end{example}

\section{Abstract Postnikov towers}\label{s:system}
The goal of this section is to give an axiomatic description of a decomposition of an object in a nice $\infty$-category, together with the data of `$k$-invariants', analogous the Postnikov tower of a space.

\begin{defn}\label{d:abstract Postnikov tower}
Let $\V$ be an $\infty$-category with finite limits. An \textbf{abstract Postnikov tower} of an object $X$ in $\V$ consists of the following data:
\begin{enumerate}
\item
An infinite tower 
$$\xymatrix@C=15pt{
X \ar[r] & ... \ar[r] & \P_a \ar[r] & ... \ar[r] & \P_1 \\
}$$
of objects $\P_a \in \V$ under \(X\) for $a \geq 1$, exhibiting \(X\) as the limit of \(\{\P_a\}_{a \geq 1}\).
\item
For each $a \geq 2$, an object $\K_a\colon \cS^{\fin}_* \to \V$ in $\T\V$ together with a Cartesian square
$$ \xymatrix{
\P_{a} \ar[r]\ar[d] & \pi(\K_{a}) \ar[d]^0 \\
\P_{a-1} \ar_-{k_{a}}[r] & \Omega^\infty\K_{a} \\
}$$
exhibiting $\P_{a} \lrar \P_{a-1}$ as a square zero extension (Section \ref{s:square zero}). 
\end{enumerate}
Note that the convention to start at $a=1$ is rather arbitrary, and in various cases it can be more natural to start at $a=0$.
\end{defn}
\begin{example}
The motivating example of an abstract Postnikov tower is the usual Postnikov tower of a space $X$. In this case, the $\K_{a}$ are given by the (suspended) parametrized Eilenberg--Maclane spectra $\K_a = \Sigma^{a+1}\rH\pi_{a}(X)$ over $\tau_{\leq 1}X$. We will come back to this in Example/Proposition \ref{e:postnikov for spaces}.
\end{example}
\begin{cons}
The data of an object $X$ equipped with an abstract Postnikov tower can be encoded by a single diagram $\T\colon\E\lrar \V$. To see this, consider for each \(a \geq 2\) the $\infty$-category
\[\ET_a := \cS^{\fin}_\ast \coprod_{\Del^{\{1\}} \times \Del^{\{a, a-1\}}} \Del^1 \times \Del^{\{a,a-1\}},\] 
where \(\Del^{\{1\}} \times \Del^{\{a, a-1\}}\) embeds in \(\cS^{\fin}_\ast\) as the arrow \(\ast \lrar S^0\). Furthermore, define
\[\ET := \Big[\ET_1 \coprod_{\Del^{\{2\}}} \ET_{2} \coprod_{\Del^{\{3\}}} \ET_3 \coprod ...\Big]\coprod_{\ZZ^{\op}_{\geq 1}} (\ZZ^{\op}_{\geq 1})^{\triangleleft}  .\]
We then observe that:
\begin{itemize}[leftmargin=*]
\item the restriction of \(\T\) to \(\cS^{\fin}_\ast \subseteq \ET_a\) corresponds to \(\K_a\).
\item the restriction of \(\T\) to \(\Del^1 \times \Del^{\{a,a-1\}} \subseteq \ET_a\) corresponds to the square
\[ \xymatrix{
\P_{a} \ar[r]\ar[d] & \K_a(\ast)=\pi(\K_a) \ar@<-5ex>[d] \\
\P_{a-1} \ar_-{k_a}[r] & \K_a(S^0)=\Omega^{\infty}(\K_a). \\
}\]
\item the restriction of \(\Phi\) to \((\ZZ_{\geq 1}^{\op})^{\triangleleft} \subseteq \ET\) encodes the tower \(X\rar \dots\rar \P_a\rar P_{a-1}\rar \dots\)
\end{itemize}
Let us define the \textbf{$\infty$-category of objects equipped with abstract Postnikov towers} to be the full subcategory
$$
\Tow(\V)\subseteq \Fun(\ET, \V)
$$
of diagrams \(\T\) for which (a) the restriction to each \(\cS^{\fin}_\ast \subseteq \ET_a\) is excisive, (b) the restriction to each \(\Del^1 \times \Del^{\{a,a-1\}} \subseteq \ET_a\) is Cartesian and (c) the restriction to \((\ZZ^{\op}_{\geq 1})^{\triangleleft}\) is a limit cone.
\end{cons}

\begin{rem}\label{r:tow-pullback}
The conditions determining \(\Tow(\V)\) inside \(\Fun(\ET,\V)\) assert that certain designated squares in \(\ET\) are sent to pullback squares and a certain \((\ZZ^{\op}_{\geq 1})^{\triangleleft}\)-diagram is a limit diagram. In particular, \(\Tow(\V)\subseteq \Fun(\ET, \V)\) is closed under limits.
\end{rem}

Note that evaluating a tower at its cone point \(\infty \in (\ZZ^{\op}_{\geq 1})^{\triangleleft} \subseteq \ET\) determines a limit-preserving functor $\ev_{\infty}\colon \Tow(\V)\lrar \V$.

\begin{defn}\label{d:functorial abstract Postnikov tower}
An \textbf{abstract Postnikov tower} on an $\infty$-category \(\V\) is defined to be a section \(\Phi\colon \V \lrar \Tow(\V)\) of the functor \(\ev_{\infty}: \Tow(\V)\lrar \V\).
\end{defn}
\begin{warning}
Note the distinction between an abstract Postnikov tower for an \emph{object in} an $\infty$-category $\V$ (Definition \ref{d:abstract Postnikov tower}) and an abstract Postnikov tower \emph{on} an $\infty$-category $\V$: the former is a single diagram in $\V$, while the latter is a family of diagrams depending functorially on $X\in \V$. This should not cause any confusion, since it is always clear from the context if we are dealing with a functor on $\V$.
\end{warning}

\begin{define}\label{d:multiplicative-quillen-tower}
Let $\V$ be an SM $\infty$-category and endow $\Fun(\ET, \V)$ with the levelwise tensor product. A \textbf{multiplicative abstract Postnikov tower} on $\V$ is a lax symmetric monoidal section \(\Phi\colon \V \lrar \Fun(\ET, \V)\) of \(\ev_\infty\colon \Fun(\ET, \V) \lrar \V\) with the property that the underlying functor of \(\Phi\) is an abstract Postnikov tower (Definition \ref{d:abstract Postnikov tower}). 
\end{define}
\begin{rem}
In general, the objectwise tensor product on $\Fun(\E, \V)$ does not induce a symmetric monoidal structure on the full sub-$\infty$-category $\Tow(\V)$.
\end{rem}
\begin{rem}
Suppose that $\Phi\colon \V\lrar \Fun(\ET, \V)$ is a multiplicative abstract Postnikov tower. Restricting to the copy of $\cS^{\fin}_\ast\subseteq \ET_m$ in level $m$, one obtains a lax monoidal functor $\K_n(\Phi)\colon \V\lrar \Fun(\cS^{\fin}_\ast, \V)$ taking values in $\T\V\subseteq \Fun(\cS^{\fin}_\ast, \V)$. Since $\T\V$ is a monoidal localization of $\Fun(\cS^{\fin}_\ast, \V)$, each $\K_n(\Phi)$ defines a lax monoidal functor $\V\lrar \T\V$ to the tangent bundle, equipped with the square zero monoidal structure (Definition \ref{d:square zero tensor product}).
\end{rem}
\begin{example}\label{e:multiplicative tower Cartesian case}
Suppose that the monoidal structure on $\V$ is given by the Cartesian product. Then the monoidal structure on $\Fun(\ET, \V)$ is the Cartesian monoidal structure as well. Since any functor between Cartesian monoidal categories has a unique lax symmetric monoidal structure (cf.\ \cite[\S 2.4.1]{Lur14}), it follows that any functorial abstract Postnikov tower has a unique multiplicative structure.
\end{example}
The main point of multiplicative abstract Postnikov towers is that they induce such towers on categories of algebras:
\begin{pro}\label{p:quillen-for-algebras}
Let $\O$ be an $\infty$-operad and let $\V$ be a symmetric monoidal $\infty$-category equipped with a multiplicative abstract Postnikov tower $\Phi\colon \V \lrar \Fun(\ET,\V)$. Then the induced map
\[ \Alg_{\O}(\V) \x{\Phi_\ast}{\lrar} \Alg_{\O}\big(\Fun(\ET, \V)\big) \simeq \Fun\big(\ET, \Alg_{\O}(\V)\big) \]
is also a multiplicative abstract Postnikov tower.
\end{pro}
\begin{proof}
First, note that we can view $\E$ as an $\infty$-operad (with only unary operations), so that $\Fun(\E, \V)\simeq \Alg_{\E}(\V)$. The symmetry of the Boardman-Vogt tensor product of $\infty$-operads \cite[Proposition 2.2.5.13]{Lur14} then induces a commuting diagram of symmetric monoidal $\infty$-categories 
\[\xymatrix{
\Alg_{\O\otimes_{\mathrm{BV}}\E}(V)\simeq\Alg_{\O}\big(\Fun(\ET, \V)\big) \ar[r]^{\simeq}\ar[d]_{\Alg_{\O}(\ev_\infty)} & \Fun\big(\ET, \Alg_{\O}(\V)\big)\simeq \Alg_{\E\otimes_{\mathrm{BV}}\O}(V)\ar[d]^{\ev_{\infty}}\\
\Alg_{\O}(\V)\ar[r]_{=} & \Alg_{\O}(\V)
}\]
in which the horizontal arrows are equivalences. It follows that $\Phi_*=\Alg_\O(\Phi)$ defines a lax symmetric monoidal section of $\ev_\infty$. To see that $\Phi_*$ takes values in the full sub-$\infty$-category $\Tow(\Alg_{\O}(\V))\subseteq \Fun(\ET, \Alg_{\O}(\V))$, consider the commuting diagram
$$\xymatrix{
\Alg_{\O}(\V)\ar[r]^-{\Phi_*}\ar[d] & \Alg_{\O}(\Fun(\ET, \V))\ar[r]^\simeq\ar[d] & \Fun(\ET, \Alg_{\O}(\V))\ar[d]\\
\Fun(\O_{\<1\>}, \V)\ar[r]_-{\Phi_*} & \Fun(\O_{\<1\>}, \Fun(\ET, \V))\ar[r]_\simeq & \Fun(\ET, \Fun(\O_{\<1\>}, \V))
}$$
where $\O_{\<1\>}$ is the underlying $\infty$-category of $\O$ \cite[Remark 2.1.1.25]{Lur14}.
Since the vertical functors preserve limits and detect equivalences, the top horizontal composite defines an abstract Postnikov tower if and only if the bottom horizontal composite does (since an $\ET$-diagram is an abstract Postnikov tower if it sends certain sub-diagrams to limit diagrams). But for the bottom horizontal composite this is clear, since limits are computed pointwise.
\end{proof}

\subsection{Examples}
Together with Proposition \ref{p:quillen-for-algebras}, the main sources of examples of multiplicative abstract Postnikov towers are the following:
\begin{expro}\label{e:postnikov for spaces}
Let $\cS$ be the $\infty$-category of spaces. Then the Postnikov tower $X\rar \dots \rar \tau_{\leq 2}(X) \rar \tau_{\leq 1}(X)$, together with its $k$-invariants, gives rise to a multiplicative abstract Postnikov tower on $(\cS, \times)$.
\end{expro}
\begin{proof}
Since we consider $\cS$ with the Cartesian monoidal structure, Example \ref{e:multiplicative tower Cartesian case} shows that it suffices to construct the Postnikov tower as an abstract Postnikov tower (i.e.\ we can forget about the multiplicative structure). The desired functorial abstract Postnikov tower can be produced at the level of simplicial sets (and is classical, cf.\ \cite{DK88, GJ}). Indeed, for every Kan complex $X$, let us define:
\begin{itemize}[leftmargin=*]
\item $\P_a(X)=\cosk_{a+1}(X)$ its $(a+1)$-coskeleton. In particular, $\P_1(X)$ is the nerve of the fundamental groupoid $\Pi_1(X)$ of $X$.

\item for every $a\geq 2$, there is a functor $\pi_a(X)\colon \Pi_1(X)\lrar \Ab$ sending a vertex of $X$ to the corresponding homotopy group. For every pointed simplicial set $S$, taking the free reduced $\pi_a(X)$-module on its simplices yields a simplicial $\Pi_1(X)$-set $\pi_a(X)\otimes S\colon \Pi_1(X)\lrar \sSet$. 

\item Recall that there is a classifying space functor $(-)_{h\Pi_1(X)}$ from $\Fun(\Pi_1(X), \sSet)$ to simplicial sets, given by the following explicit point-set model for the homotopy colimit: $Y_{h\Pi_1(X)}$ has $n$-simplices given by tuples of $x_0\rar \dots\rar x_n$ in $\Pi_1(X)$ and an $n$-simplex of $Y(x_0)$. In particular, $(\ast)_{h\Pi_1(X)}=\P_1(X)$ is the nerve of the fundamental groupoid.

\item Let us define $\sSet^{\fin}_\ast$ to be the full subcategory of pointed simplicial sets whose image in the $\infty$-category $\cS_\ast$ of pointed spaces is finite. We then define $\K_{X, a}\colon \sSet^{\fin}_\ast\lrar \sSet$ by 
$$
\K_{X,a}(T)= \Big(\pi_a(X)\otimes (T\wedge S^{a+1})\Big)_{h\Pi_1(X)}
$$ 
where $S^{a+1}=\Delta^{a+1}/\mathrm{sk}_{a}\Delta^{a+1}$. 

\item By \cite{DK88}, there is a natural map of simplicial sets  for each $a\geq 2$
$$
k_{a}\colon \P_{a-1}(X)\lrar \K_{X, a}(S^0).
$$
By construction, the map $k_{a}$ is trivial on all $(a+1)$ simplices in $\P_{a-1}(X)$ that arise as the image of an $(a+1)$-simplex in $\P_{a}(X)$, so that there is a commuting square
\begin{equation}\label{d:kinvariant square}\vcenter{\xymatrix{
\P_{a}(X)\ar[d]\ar[r] & \K_{X, a}(\ast)\ar[d]^0\\
\P_{a-1}(X)\ar[r]_-{k_{a}} & \K_{X, a}(S^0).
}}\end{equation}
\end{itemize}
For any Kan complex $X$, the functor $\K_{X, a}$ preserves weak equivalences of simplicial sets and hence determines a functor of $\infty$-categories $\K_{X, a}\colon \cS^{\fin}_{\ast}\lrar \cS$. It is straightforward to verify the conditions of Proposition \ref{p:gamma objects as spectra}, which imply that $\K_{X, a}$ is excisive because $\cS$ admits loopspace machinery (Example \ref{ex:loopspace}). Furthermore, the square \eqref{d:kinvariant square} defines a pullback square in the $\infty$-category $\cS$ by \cite{DK88} and the sequence $X\rar \dots\rar \P_a(X)\rar \P_{a-1}(X)\rar \dots$ is a homotopy limit sequence. It follows that the above construction defines, for every Kan complex $X$, a simplicial model for an abstract Postnikov tower in the $\infty$-category $\cS$.

All of the above data is strictly functorial in maps of Kan complexes and sends weak equivalences to pointwise weak equivalences of simplicial sets. It therefore defines a section
$$\xymatrix@C=3pc{
\cS\ar@{..>}[r] & \Tow(\cS)\ar@/_1pc/[l]_{\ev_\infty}
}$$
on $\infty$-categorical localizations, as desired.
\end{proof}
\begin{example}\label{e:algebraic theories}
The abstract Postnikov tower of Example \ref{e:postnikov for spaces} has an additional feature: the explicit construction given in the proof preserves finite products of simplicial sets, so the resulting functor $\Phi\colon \cS\lrar \Fun(\ET, \cS)$ preserves finite products as well. It follows that for any small $\infty$-category with finite products $\T$, the $\infty$-category $\Fun^{\times}(\T, \cS)$ of product-preserving functors comes with an abstract Postnikov tower
$$\xymatrix{
\Fun^\times(\T, \cS)\ar[r]^-{\Phi_*} & \Fun^\times(\T, \Fun(\ET, \cS))\simeq \Fun(\ET, \Fun^\times(\T, \cS)).
}$$
For every $A\in \Fun^\times(\T, \cS)$, this provides a refinement of the tower $A\rar \dots\rar \tau_{\leq 2}A\rar \tau_{\leq 1}A$ of truncations of $A$. In particular, when $\T$ is an algebraic theory, this shows that $\T$-algebras over in $\cS$ have Postnikov towers equipped with $k$-invariants (cf.\ \cite{GH00} for algebras over simplicial operads).
\end{example}
\begin{example}\label{e:toposes}
Let $\X$ be an $\infty$-topos in which Postnikov towers converge \cite[Definition 5.5.6.23]{Lur09}, i.e.\ $\X\lrar \lim_n \tau_{\leq n}\X$ is the limit of its full subcategories of truncated objects (this implies that $\X$ is hypercomplete). In this case, there exists a reflective localization $L\colon \Fun(\C^{\op}, \cS)\adj \X\colon i$ such that $L$ is left exact and preserves (limits of) Postnikov towers. We then obtain an abstract Postnikov tower on $\X$
$$\xymatrix@C=1.8pc{
\X\ar[r]^-i & \Fun(\C^{\op}, \cS)\ar[r]^-{\Phi_*} & \Fun(\C^{\op}, \Fun(\ET, \cS))\simeq \Fun(\ET, \Fun(\C^{\op}, \cS))\ar[r]^-{L_*} & \Fun(\ET, \X).
}$$
Indeed, this sends every object $X\in \X$ to the abstract Postnikov tower of the presheaf $i(X)$ (applying Example \ref{e:postnikov for spaces} pointwise in $\C$), and then applies $L$ to the resulting diagram of presheaves. Since $L$ is left exact and preserves Postnikov towers, the resulting $\ET$-diagram in $\X$ is indeed an abstract Postnikov tower.
\end{example}
\begin{obs}\label{obs:abelian homotopy groups case}
The proof of Example \ref{e:postnikov for spaces} admits the following modification: let $\cS^{\pi-\ab}\subseteq \cS$ be the full subcategory consisting of those spaces $X$ such that each homotopy group $\pi_1(X, x)$ is abelian and acts trivially on the higher $\pi_n(X, x)$. Then there exists a multiplicative abstract Postnikov tower
$$\xymatrix{
\cS^{\pi-\ab}\ar[r] & \Tow(\cS^{\pi-\ab})\subseteq \Tow(\cS)
}$$
whose value on a space $X$ is the Postnikov tower $X\rar \dots\rar \tau_{\leq 1}X\rar \pi_0(X)$ including the zeroth stage. Furthermore, the $k$-invariants are given by maps 
$$\xymatrix{
k_{a}\colon \tau_{\leq a-1}X\ar[r] & \Omega^\infty(\K_{a}(X))
}$$
where $\K_{a}(X)$ is the parametrized spectrum over $\pi_0(X)$ whose fiber over $x\in\pi_0(X)$ denotes the suspended Eilenberg--Maclane spectrum $H(\pi_{a}(X, x))[a+1]$.
\end{obs}
\begin{example}\label{e:additive case}
Let $\Mon_{\EE_\infty}(\cS^{\pi-\ab})$ be the $\infty$-category of $\EE_\infty$-spaces whose underlying space has trivial actions of $\pi_1$. Proposition \ref{p:quillen-for-algebras} and Observation \ref{obs:abelian homotopy groups case} imply that the Postnikov tower $A\rar \dots \rar \tau_{\leq 1}A\rar \tau_{\leq 0}A$ is part of a multiplicative abstract Postnikov tower $\Phi^{\ab}$ on $\big(\Mon_{\EE_\infty}(\cS^{\pi-\ab}), \times\big)$.

Let $A$ be a grouplike $\EE_\infty$-space. Then $A$ is in particular contained in $\Mon_{\EE_\infty}(\cS^{\pi-\ab})$. The corresponding abstract Postnikov tower $\Phi^{\ab}_A\colon \ET\lrar \Mon_{\EE_\infty}(\cS^{\pi-\ab})$ has the property that $\pi_0(\Phi_A)$ is the constant diagram with value $\pi_0(A)$. In particular, $\Phi^{\ab}_A$ takes values in the full subcategory of grouplike $\EE_\infty$-monoids. It follows that the multiplicative abstract Postnikov tower $\Phi^{\ab}$ restricts to a multiplicative abstract Postnikov tower on grouplike $\EE_\infty$-monoids, which fits into a commuting square
$$\xymatrix{
\Grp_{\EE_\infty}(\cS)\ar[r]^-{\Phi^{\ab}}\ar[d]_-{\mathrm{forget}} & \Tow\big(\Grp_{\EE_\infty}(\cS)\big)\ar[d]^-{\textrm{forget}}\\
\cS^{\pi-\ab}\ar[r]^-{\text{\footnotesize \ref{obs:abelian homotopy groups case}}} & \Tow(\cS).
}$$
\end{example}
\begin{expro}\label{e:postnikov for prestable}
Let $\V$ be a complete Grothendieck prestable $\infty$-category. Then the towers $X\rar \dots \rar \tau_{\leq 1}(X) \rar \tau_{\leq 0}(X)$ are part of an abstract Postnikov tower on $\V$. If $\V$ is furthermore presentable SM, then this is a multiplicative abstract Postnikov tower.
\end{expro}
\begin{rem}
The above conditions are equivalent to $\V=\W^{\geq 0}$ being the connective part of a left complete $t$-structure on a stable presentable $\infty$-category $\W$ with the property that $\tau_{\geq 0}\colon \W\lrar \W^{\geq 0}$ preserves filtered colimits \cite[Proposition C.1.4.1]{Lur16}.
\end{rem}
\begin{proof}
To avoid repetition, let us prove the claim for a symmetric monoidal $\V$; the simpler non-monoidal case is proven in the same way, removing all references to the monoidal structure. To begin, since $\V$ is presentable there exists a small full subcategory $\C^{\op}\subseteq \V$ which generates $\V$ under colimits; without loss of generality, we can assume that $\C^{\op}$ is closed under finite direct sums and finite tensor product. This determines adjunctions
$$\xymatrix@C=2.8pc{
L\colon \Fun(\C, \cS)\ar@<1ex>[r]^{L_1} & \Fun^{\times}(\C, \cS)\ar@<1ex>[l]^-{i_1} \ar@<1ex>[r]^-{L_2} & \V\ar@<1ex>[l]^-{i_2} \colon i
}$$
where the middle $\infty$-category consists of product-preserving functors $\C\lrar \cS$ and the functors $i_1$ and $i_2$ are both fully faithful. The composite adjoint pair $(L, i)$ realizes $\V$ as a symmetric monoidal localization of $\Fun(\C, \cS)$, endowed with the Day convolution product \cite[Proposition 2.2]{NS}. Furthermore, $L_2$ is left exact by the $\infty$-categorical Gabriel-Popescu theorem \cite[Theorem C.2.1.6]{Lur16}. 

Note that $\Fun^{\times}(\C, \cS)$ is an additive $\infty$-category, so that the canonical forgetful functor $\Fun^{\times}\big(\C, \Grp_{\EE_\infty}(\cS)\big)\lrar \Fun^{\times}(\C, \cS)$ is an equivalence. In particular, every $F\in \Fun^{\times}(\C, \cS)$ defines a diagram with values in the subcategory $\cS^{\pi-\ab}$ of Observation \ref{obs:abelian homotopy groups case}. Using this, we can now consider the following commuting diagram:
\begin{equation}\label{d:postnikov for prestable}\begin{split}\scalebox{0.92}{\xymatrix@R=0.5pc@C=0.9pc{
\hspace{20pt}\V\hspace{-2pt}\ar[0, 0]+<16pt, -4pt>;[rd]^-{i_2} \\
& \Fun^{\times}(\C, \cS)\ar[rr]^-{\Phi^{\times}_*}\ar[dd]_-{i_1} & & \Fun^{\times}(\C, \Fun(\ET, \cS))\ar[rr]^{\simeq}\ar[dd] & & \Fun(\ET, \Fun^\times(\C, \cS))\ar[dd]^-{(i_1)_*}\ar@/^1pc/[rddd]^-{(L_2)_*}\\
\\
& \Fun(\C, \cS^{\pi-\ab})\ar[rr]_-{\Phi_*} & & \Fun(\C, \Fun(\ET, \cS))\ar[rr]_-{\simeq} & & \Fun(\ET, \Fun(\C, \cS))\ar[rd]_-{L_*}\\
& & & & & & \hspace{-3pt}\Fun(\ET, \V).
}}\end{split}\raisetag{-50pt}\end{equation}
Here $\Phi_*$ applies the (multiplicative) abstract Postnikov tower $\Phi$ from Observation \ref{obs:abelian homotopy groups case} pointwise, $\Phi^{\times}_*$ is simply the restriction to product-preserving functors (Example \ref{e:algebraic theories}) and the right triangle commutes since $L\circ i_1\simeq L_2\circ L_1\circ i_1\simeq L_2$. Let us denote the total composite functor by $\Psi$.

Let us first prove that $\Psi$ is a lax symmetric monoidal section of $\ev_\infty\colon \Fun(\ET, \V)\lrar \V$. To see this, note that $\Psi$ is the top horizontal composite of
$$\scalebox{0.92}{\xymatrix@C=2pc{
\V\ar[r]^-{i}\ar[d]_= & \Fun(\C, \cS^{\pi-\ab})\ar[d]\ar[r]^-{\Phi_\ast} & \Fun(\C, \Fun(\ET, \cS))\ar[d]^{\ev_\infty} \ar[r]^{\simeq} &\Fun(\ET, \Fun(\C, \cS))\ar[d]^{\ev_\infty}\ar[r]^-{L_*} & \Fun(\E, \V)\ar[d]^{\ev_\infty}\\
\V\ar[r]_-{i} & \Fun(\C, \cS)\ar[r]_-{=} & \Fun(\C, \cS)\ar[r]_-{=} & \Fun(\C, \cS)\ar[r]_-L & \V.
}}$$
All $\infty$-categories carry natural symmetric monoidal structures, given levelwise with respect to $\ET$ and by Day convolution with respect to $\C$. All functors in this diagram are then lax symmetric monoidal. Indeed, for $\Phi_*$ this uses that Day convolution is functorial with respect to lax symmetric monoidal functors in the target (see \cite[\S 2.2.6]{Lur14}), while for $L_*$ it uses that the levelwise tensor product is functorial with respect to (lax) symmetric monoidal functors. Furthermore, the equivalence in the top row is part of a symmetric monoidal equivalence
$$
\Fun\big(\C, \Fun(\ET, \cS)_\text{lev}\big)_\text{Day}\simeq \Fun\big(\ET, \Fun(\C, \cS)_\text{Day}\big)_\text{lev}.
$$
Finally, since $(L, \iota)$ defines a symmetric monoidal localization, the bottom composite is naturally equivalent to the identity as a lax symmetric monoidal functor. It follows that $\Psi$ indeed defines a lax symmetric monoidal section of $\ev_\infty$.

It remains to verify that for every object $X\in \V$, the diagram $\Psi_X\colon \ET\lrar \V$ indeed satisfies the conditions of Definition \ref{d:abstract Postnikov tower}. Realizing $\Psi$ as the top composite in \eqref{d:postnikov for prestable}, we see that $\Psi_X$ is given by
\begin{equation}\label{d:decomposition of Psi}\xymatrix@C=2.8pc{
\Psi_X\colon \ET\ar[r]^-{\Phi^{\times}_{i_2(X)}} & \Fun^{\times}(\C, \cS)\ar[r]^-{L_2}  & \V.
}\end{equation}
Here $\Phi^{\times}$ is the abstract Postnikov tower provided (pointwise in $\C$) by Example \ref{e:algebraic theories}. 

Note that $\Psi_X$ sends all designated squares in $\ET$ to pullback squares in $\V$ (Remark \ref{r:tow-pullback}) since $L_2$ is left exact. The underlying tower of $\Phi_X$ is given by
$$\xymatrix@C=1.3pc{
L_2i_2(X)\ar[r] & \dots \ar[r] & L_2\big(\tau_{\leq a}i_2(X)\big)\ar[r] & L_2\big(\tau_{\leq a-1}i_2(X)\big)\ar[r] & \dots\ar[r] & L_2\big(\tau_{\leq 0}i_2(X)\big).
}$$
Since $L_2$ is a left exact left adjoint, it commutes with the truncation functors \cite[Proposition 5.5.6.28]{Lur09}. It follows that the tower underlying $\Psi_X$ agrees with the Postnikov tower of $X$ in $\V$. Since Postnikov towers in $\V$ converge (since $\V$ was assumed to be complete), this tower is a limit tower. We conclude that $\Psi$ provides a multiplicative functorial abstract Postnikov tower.
\end{proof}
\begin{rem}
In the setting of Example \ref{e:postnikov for prestable}, the proof shows that the $k$-invariants fit into \emph{pushout} squares
$$\xymatrix{
\tau_{\leq a}X\ar[d]\ar[r] & \pi_0(X)\ar[d]\\
\tau_{\leq a-1}X\ar[r]_-{k_{a}} & \Omega^\infty(\K_{a}).
}$$
Note that this implies that $\Omega^{\infty}(\K_{a})\simeq \pi_0(X)\oplus \Sigma^{a+1}\pi_{a}(X)$. 
The main content of Example \ref{e:postnikov for prestable} is that these $k$-invariants are compatible with the tensor product; in particular, Proposition \ref{p:quillen-for-algebras} shows that for an $\O$-algebra $A$ in $\V$, the $k$-invariant defines an $\O$-algebra map into the square zero extension $\pi_0(A)\oplus \Sigma^{a+1}\pi_{a}(A)$ (cf.\ \cite[Corollary 7.4.1.28]{Lur14}).
\end{rem}
\begin{example}\label{e:algebras in prestable}
Let $\V$ be a complete SM Grothendieck prestable $\infty$-category and let $\O$ be an $\infty$-operad. For example, one can take $\V=\Sp^{\geq 0}$ to be the $\infty$-category of connective spectra with the smash product. Combining Proposition \ref{p:quillen-for-algebras} and Example/Proposition \ref{e:postnikov for prestable}, we find that the Postnikov tower $A\lrar \dots \lrar \tau_{\leq 1}A\lrar \tau_{\leq 0} A$ of $\O$-algebras in $\V$ is part of a (functorial, multiplicative) abstract Postnikov tower on $\Alg_{\O}(\V)$. By Example \ref{e:square zero extensions in stable}, this means that each stage of the Postnikov tower fits into a pullback square of $\O$-algebras
$$\xymatrix{
\tau_{\leq a}A\ar[d]\ar[r] & \pi_0(A)\ar[d]^-{(\id, 0)}\\
\tau_{\leq a-1}A\ar[r]_-{k_a} & \pi_0(A)\oplus \Sigma^{a+1}\pi_a(A)
}$$
where $\pi_0(A)\oplus \Sigma^{a+1}\pi_a(A)$ is the square zero extension of $\pi_0(A)$ by the operadic module $\Sigma^{a+1}\pi_a(A)$. By specializing to $\O=\EE_n$, this recovers \cite[Corollary 7.4.1.28]{Lur14}.
\end{example}

\section{Abstract Postnikov towers of enriched categories}\label{s:ncat}
In the previous section we have seen how multiplicative abstract Postnikov towers give rise to multiplicative abstract Postnikov towers on $\infty$-categories of algebras over operads (Proposition \ref{p:quillen-for-algebras}). The purpose of this section is to prove that similarly, a multiplicative abstract Postnikov tower on a symmetric monoidal $\infty$-category $\V$ induces a multiplicative abstract Postnikov tower on the $\infty$-category of $\V$-enriched $\infty$-categories.

\subsection{Recollections on enriched $\infty$-categories}
Let us briefly recall some elements of the theory of enriched $\infty$-categories developed by Gepner--Haugseng \cite{GH15}. Let $\V$ be a symmetric monoidal $\infty$-category. For every space $X$, there exists a (nonsymmetric) $(X\times X)$-coloured $\infty$-operad $\O_X$ \cite[Definition 4.2.4]{GH15} whose algebras can be informally described as follows: an algebra consists of objects $\Map(x, y)\in \V$, depending functorially on $(x, y)\in X\times X$, together with composition operations satisfying obvious associativity conditions. We will refer to such algebras as \textbf{$\V$-enriched categorical algebras with space of objects} $X$ and denote the $\infty$-category of these by $\Alg_{\Cat}^X(\V)$.

The $\infty$-category $\Alg_{\Cat}^X(\V)$ depends (contravariantly) functorially on $X$ and for each $f\colon X\lrar Y$, the restriction functor $f^*\colon \Alg_{\Cat}^Y(\V)\lrar \Alg_{\Cat}^X(\V)$ preserves limits \cite[Corollary 3.2.2.4]{Lur14}. One defines the $\infty$-category of \textbf{categorical algebras}
\begin{equation}\label{e:categorical algebras fibration}\xymatrix{
\Ob\colon \Alg_{\Cat}(\V) = \int_{X\in \cS} \Alg_{\Cat}^X(\V)\ar[r] & \cS
}\end{equation}
to be the domain of the corresponding Cartesian fibration \cite[Definition 4.3.1]{GH15}. If $\V$ is a presentable monoidal $\infty$-category, then $\Alg_{\Cat}(\V)$ is presentable as well. The $\infty$-category $\Alg_{\Cat}(\V)$ is functorial in $\V$ with respect to lax monoidal functors; in fact, the resulting functor from the $\infty$-category of monoidal $\infty$-categories and lax monoidal functors
$$\xymatrix{
\Cat_\infty^{\otimes, \text{lax}}\ar[r] & \Cat_\infty; \hspace{4pt} \V\ar@{|->}[r] & \Alg_{\Cat}(\V)
}$$
is itself a lax monoidal functor, with respect to the Cartesian product of monoidal $\infty$-categories \cite[Proposition 4.3.11]{GH15}. Consequently, when $\V$ is (presentable) symmetric monoidal, there exists a (presentable) symmetric monoidal structure on $\Alg_{\Cat}(\V)$ \cite[Corollary 4.3.12]{GH15}; informally, given two categorical algebras $\CC$, $\DD$ with spaces of objects $X, Y$, their tensor product $\CC\otimes \DD$ has space of objects $X\times Y$ and mapping objects
$$
\Map_{\CC\otimes\DD}\big((x_0, y_0), (x_1, y_1)\big)=\Map_{\CC}(x_0, x_1)\otimes \Map_{\DD}(y_0, y_1).
$$
In particular, the unit is given by the categorical algebra $[0]_{1_\V}$ with a single object $\ast$ and with $1_\V$ as endomorphisms.

The $\infty$-category of \textbf{$\V$-enriched $\infty$-categories} $\Cat(\V)$ is defined to be the full subcategory $\Cat(\V)\subseteq \Alg_{\Cat}(\V)$ of \textbf{complete} categorical algebras. More precisely, there is a functor 
$$
J[-]\colon \Delta\lrar \Alg_{\Cat}(\V)
$$
sending $[n]$ to the categorical algebra with object set $\{0, \dots, n\}$, all mapping objects being \(1_{\V}\) and all compositions being equivalences. We will abbreviate $J=J[1]$. Every categorical algebra $\CC$ then defines a simplicial space
$$\xymatrix{
\Delta^{\op}\ar[r] & \cS; \hspace{4pt} [n]\ar@{|->}[r] & \Map_{\Alg_{\Cat}(\V)}\big(J[n], \CC\big).
}$$
This simplicial space is a Segal groupoid \cite[Corollary 5.2.7]{GH15} and $\CC$ is defined to be complete if this Segal groupoid is essentially constant. Note that the above Segal space only depends on the categorical algebra in $\cS$ obtained by applying the lax monoidal functor $\Map(1_\V, -)\colon \V\lrar \cS$ to all mapping objects. Furthermore, the space $\Map(J[n], \CC)\subseteq \Map([n]_{1_\V}, \CC)$ is a union of path components in the space of $n$-composable sequences of arrows in $\CC$ \cite[Proposition 5.1.17]{GH15}.

The inclusion of $\V$-enriched $\infty$-categories into categorical algebras is part of an adjoint pair
$$\xymatrix@C=3pc{
\Alg_{\Cat}(\V)\ar@<1ex>[r]^-{(-)^\wedge} & \hspace{2pt}\Cat(\V)\ar@{_{(}->}@<1ex>[l]_-\perp
}$$
whose left adjoint is called \textbf{completion} \cite[Theorem 5.6.6]{GH15}. When $\V$ is symmetric monoidal, this is a symmetric monoidal localization \cite[Proposition 5.7.14]{GH15}. 

Finally, let us recall that the completion functor realizes $\Cat(\V)$ as the localization of $\Alg_{\Cat}(\V)$ at the \textbf{Dwyer--Kan equivalences}, i.e.\ the fully faithful and essentially surjective functors in the following sense:
\begin{defn}
We will say that a map of categorical algebras $f\colon \CC\lrar\DD$ is:
\begin{enumerate}[leftmargin=*]
\item \textbf{fully faithful} if for every two objects $x, y\in \Ob(\CC)$, the map
$$
f\colon \Map_{\CC}(x, y)\lrar \Map_{\DD}\big(f(x), f(y)\big)
$$
is an equivalence in $\V$. Equivalently, $f$ is a Cartesian arrow for the Cartesian fibration \eqref{e:categorical algebras fibration}.
\item \textbf{essentially surjective} if the map
$$
\Map(\{0\}, \CC)\times_{\Map(\{0\}, \DD)}\Map(J, \DD)\lrar \Map(\{1\}, \DD)
$$
is surjective on \(\pi_0\). Here the mapping spaces are taken in the \(\infty\)-category \(\Alg_{\Cat}(\V)\).
\item an \textbf{isofibration} if the induced map
\[ \Map(J,\CC) \lrar \Map(J,\DD) \times_{\Map(\{0,1\},\DD)} \Map(\{0,1\},\CC) \]
is surjective on \(\pi_0\).
\end{enumerate}
\end{defn}

\subsection{The cube and tower lemmas}
Throughout, let $\V$ be a monoidal $\infty$-category. The purpose of this section is to record two kinds of (`homotopy') limits of categorical algebras that are preserved by the completion functor $(-)^\wedge\colon \Alg_{\Cat}(\V)\lrar \Cat(\V)$. The results and arguments are very analogous to the usual way of computing homotopy limits of categories in terms of the canonical model structure on categories.
\begin{lem}\label{l:chase}
Consider a commutative square of categorical algebras
\begin{equation}\label{e:square} 
\vcenter{\xymatrix{
\CC' \ar[r]^-{g'}\ar[d]_-{p} & \DD'\ar[d]^-{q} \\
\CC \ar[r]_-{g} & \DD \\
}}
\end{equation}
such that \(g'\) is essentially surjective, \(g\) is fully faithful and \(p\) is an isofibration. Then
\[\scalebox{0.92}{$\Map(\{0\},\CC') \hspace{-2pt}\displaystyle\mathop{\times}_{\Map(\{0\},\DD')} \hspace{-2pt} \Map(J,\DD')  \lrar \Map(\{0\},\CC) \hspace{-2pt}\displaystyle\mathop{\times}_{\Map(\{0\},\DD)} \hspace{-2pt} \Map(J,\DD) \hspace{-2pt} \displaystyle\mathop{\times}_{\Map(\{1\},\DD)}\hspace{-2pt} \Map(\{1\},\DD')$}\]
is surjective on path components. 
\end{lem}
\begin{proof}
Applying the functor \(\Alg_{\Cat}(\V) \lrar \Alg_{\Cat}(\cS)\) induced by  the lax monoidal functor \(\Map_{\V}(1_{\V},-)\colon \V \lrar \cS\), we may as well assume that \(\V = \cS\). Explicitly, given objects \(c \in \CC\) and \(d'\in \DD'\) together with an equivalence \(\alp\colon g(c) \x{\simeq}{\lrar} d=q(d')\) in \(\DD\), we need to find an object $c'\in \CC'$ lying over $c$ and an equivalence \(\alp'\colon g'(c') \x{\simeq}{\lrar} d'\) in \(\DD'\) lying over \(\alp\).

To begin, since \(g'\) is essentially surjective there exists an object \(t' \in \CC'\) and an equivalence \(\gam'\colon g'(t') \x{\simeq}{\lrar} d'\) in \(\DD'\). %, where \(t'\) is the image of \(w'\). 
One can then complete \(q(\gam')\) and \(\alp\) to a commutative triangle
\begin{equation}\label{e:top-triangle}
\vcenter{\xymatrix{
& q(g'(t')) \ar[dr]^{q(\gam')} & \\
g(c)\ar[ur]^{\del}\ar[rr]^{\alp} && d\\
}}
\end{equation}
in \(\DD\) for some equivalence \(\del\colon g(c) \lrar q(g'(t')) \simeq g(p(t'))\). Since \(g\) is fully-faithful there exists an equivalence \(\eps\colon c\lrar p(t')\) lying over \(\del\), and since \(p\) is an isofibration we may lift \(\eps\) to an equivalence \(\eps'\colon c' \lrar t'\) for some \(c' \in \CC'\) lying over \(c\).  
We may then complete \(g'(\eps')\) and \(\gam'\) to a commutative diagram
\begin{equation}\label{e:bottom-triangle}
\vcenter{\xymatrix{
& g'(t') \ar[dr]^{\gam'} & \\
g'(c')\ar[ur]^{g'(\eps')}\ar[rr]^{\alp'} && d'\\
}}
\end{equation}
for some equivalence \(\alp'\colon g'(c') \lrar d'\) in \(\DD'\). Since the image of triangle~\eqref{e:bottom-triangle} under \(q\colon \DD'\lrar\DD\) agrees with triangle~\eqref{e:top-triangle} on the inner horn it follows that \(q(\alp')\) is homotopic to \(\alp\). This yields the desired data of \(c'\) and \(\alp'\colon g'(c') \lrar d'\) so that the proof is complete.
\end{proof}

\begin{lem}[Cube lemma]\label{l:cube lemma}
Consider a map of Cartesian squares in \(\Alg_{\Cat}(\V)\)
\begin{equation}\label{e:alpha} 
\vcenter{\xymatrix@R=10pt{
\PP \ar[r]\ar[dd]  & \CC'   \ar[dd]^{p} &&& \QQ \ar[r]\ar[dd] & \DD' \ar[dd]^{q} \\
&& \ar@{=>}[r]^{\left(\begin{matrix} f & g' \\ g'' & g \end{matrix}\right)} &&& \\
\CC'' \ar[r]_h & \CC  &&& \DD'' \ar[r]_k & \DD \\
}}
\end{equation}
such that \(p\) is an isofibration. If the components \(g\colon\CC\lrar \DD\), \(g'\colon\CC' \lrar \DD'\) and \(g''\colon\CC'' \lrar \DD''\) are Dwyer--Kan equivalences, then the same holds for \(f\colon\PP \lrar \QQ\).
\end{lem}
\begin{cor}\label{c:main-point-cube}
The completion functor \((-)^{\wedge}\colon \Alg_{\Cat}(\V) \lrar \Cat(\V)\) sends pullback squares with one leg being an isofibration to pullback squares. 
\end{cor}
\begin{proof}
Apply Lemma~\ref{l:cube lemma} to the case where the maps \(g,g'\) and \(g''\) exhibit \(\DD,\DD'\) and \(\DD''\) as the completions of \(\CC,\CC'\) and \(\CC''\) respectively (in this case \(\QQ\simeq \DD'\times_{\DD} \DD''\) is automatically complete).
\end{proof}

\begin{proof}[Proof of Lemma~\ref{l:cube lemma}]
To show that \(f\) is fully-faithful, let \(x,y \in \PP\) be two objects, and consider the induced map of squares
\[ \xymatrix@R=10pt@C=1.1pc{
\Map_{\PP}(x,y) \ar[rr]\ar[dd] & & \Map_{\CC'}(x,y)   \ar[dd]^{p_*} &&& \Map_{\QQ}(x,y) \ar[rr]\ar[dd] & &\Map_{\DD'}(x,y) \ar[dd]^{q_*} \\
&&& \ar@{=>}[r] &&&& \\
\Map_{\CC''}(x,y) \ar[rr] & & \Map_{\CC}(x,y)  &&& \Map_{\DD''}(x,y) \ar[rr] & & \Map_{\DD}(x,y). \\
}\]
Both squares are Cartesian in \(\V\) and by assumption the three maps associated to \(g,g'\) and \(g''\) are equivalences, so that the map \(f_*\colon \Map_{\PP}(x, y)\lrar \Map_{\QQ}(x,y)\) is an equivalence as well.

Let us now show that \(f\) is essentially surjective as a map of categorical algebras. Essential surjectivity is detected on the level of the underlying space-valued categorical algebras under the canonical map \(\Alg_{\Cat}(\V) \lrar \Alg_{\Cat}(\cS)\) (induced by the lax monoidal functor \(\Map_{\V}(1_\V,-)\)). We may hence assume that \(\V = \cS\). Let \(y \in \QQ\) be an object and let \(d' \in \DD', d\in \DD, d'' \in \DD''\) be its images.
Since \(g''\colon \CC''\lrar \DD''\) is essentially surjective there exists an object \(c'' \in \CC''\) and an equivalence \(\alp''\colon g''(c'') \x{\simeq}{\lrar} d''\) in \(\DD''\). Let \(c := h(c'') \in \CC\). 
Applying Lemma~\ref{l:chase} to the image 
\[\alp \colon g(c) \simeq k(g''(c'')) \x{\simeq}{\lrar} k(d'') \simeq d\simeq q(d')\] 
of \(\alp''\) in \(\DD\) we deduce the existence of an object \(c' \in \CC'\) lying over \(c\) and an equivalence \(\alp'\colon g'(c') \lrar d'\) in \(\DD'\) lying over \(\alp\).
The compatible triple \((c,c',c'')\) now determines an object \(x\in\PP\) while the compatible triple \((\alp,\alp',\alp'')\) determines an equivalence \(g(x)\x{\simeq}{\lrar} y\) in \(\QQ\).
\end{proof}

\begin{lem}[Tower lemma]\label{l:tower lemma}
Consider a natural transformation between limit cones in \(\Alg_{\Cat}(\V)\)
$$\xymatrix{
\PP\ar[r]\ar[d]_{f} & \dots\ar[r] & \CC_2\ar[r]^{p_2}\ar[d]_{g_2} & \CC_1\ar[d]_{g_1}\ar[r]^{p_1} & \CC_0\ar[d]^{g_0}\\
\QQ\ar[r] & \dots\ar[r] & \DD_2\ar[r]_{q_2} & \DD_1\ar[r]_{q_1} & \DD_0.
}$$
Suppose that all $p_i$ for $i\geq 1$ are isofibrations and all $g_i$ for $i\geq 0$ are Dwyer--Kan equivalences. Then $f$ is a Dwyer--Kan equivalence as well.
\end{lem}

\begin{cor}\label{c:main-point-tower}
The completion functor \((-)^{\wedge}\colon \Alg_{\Cat}(\V) \lrar \Cat(\V)\) sends limits of towers of isofibrations to limits. 
\end{cor}
\begin{proof}
Apply Lemma~\ref{l:tower lemma} to the case where the maps \(g_i\) exhibit \(\DD_i\) as the completion of \(\CC_i\) (in which case \(\QQ\) is automatically complete).
\end{proof}
\begin{proof}[Proof of Lemma~\ref{l:tower lemma}]
To show that \(f\) is fully-faithful, let \(x,y \in \PP\) be two objects, and consider the induced map of towers
\[ \xymatrix{
\Map_{\PP}(x,y) \ar[r]\ar[d]_{f_*}  & \dots\ar[r] & \Map_{\CC_1}(x,y)\ar[r]\ar[d]_{(g_1)_*} & \Map_{\CC_0}(x, y)\ar[d]^{(g_0)_*}\\
\Map_{\QQ}(x,y) \ar[r] &\dots\ar[r] & \Map_{\DD_1}(x,y) \ar[r] & \Map_{\DD_0}(x,y) \\
}.\]
Then both towers are limit towers in \(\V\) and by assumption the \((g_i)_*\) are equivalences, so that the map \(f_*\colon \Map_{\PP}(x, y)\lrar \Map_{\QQ}(x,y)\) is an equivalence as well.

Let us now show that \(f\) is essentially surjective as a map of categorical algebras. We may again assume that \(\V = \cS\). Let \(y \in \QQ\) be an object and let \(d_i \in \DD_i\) be its images. 
Since \(g_0\colon \CC_0\lrar \DD_0\) is essentially surjective, there exists an object \(c_0 \in \CC_0\) and an equivalence \(\alp_0\colon g_0(c_0) \x{\simeq}{\lrar} d_0\) in \(\DD_0\). 
Applying Lemma~\ref{l:chase} to $\alp_0$ and $d_1\in\DD_1$, we deduce the existence of an object \(c_1 \in \CC_1\) lying over \(c_0\) and an equivalence \(\alp_1\colon g_1(c_1) \lrar d_1\) in \(\DD_1\) lying over \(\alp_1\). Proceeding inductively, we obtain compatible sequences of objects \(c_i\in\CC_i\) and equivalences \(\alp_i\colon g_i(c_i)\lrar d_i\). These determine an object \(x\in\PP\) and an equivalence \(g(x)\x{\simeq}{\lrar} y\) in \(\QQ\).
\end{proof}

\subsection{Abstract Postnikov towers of enriched $\infty$-categories}
We now turn to our main result, which allows one to produce abstract Postnikov towers of $\V$-enriched $\infty$-categories from (certain) multiplicative abstract Postnikov towers on $\V$.

\begin{define}\label{d:0-connected}
Let $\V$ be a SM $\infty$-category. 
We will say that a map \(f\colon X \lrar Y\) in \(\V\) is \emph{\(0\)-connected} if the induced map \(\Map_{\V}(1_{\V},X) \lrar \Map_{\V}(1_{\V},Y)\) is \(0\)-connected as a map of spaces, that is, its homotopy fibers are all non-empty and connected. 

We will say that an abstract Postnikov tower of an object \(\T\colon \ET \lrar \V\) is \(0\)-connected if it sends every map in \(\ET\) to a \(0\)-connected map. An abstract Postnikov tower \(\Phi\colon \V \lrar \Fun(\ET,\V)\) on $\V$ is \(0\)-connected if it sends each object $X\in\V$ to a \(0\)-connected abstract Postnikov tower of $X$.
\end{define}
\begin{example}
The usual Postnikov tower on spaces (Example \ref{e:postnikov for spaces}) is a $0$-connected abstract Postnikov tower: for every space $X$, the resulting abstract Postnikov tower is constant after applying $\tau_{\leq 1}$. For more general $\infty$-toposes (Example \ref{e:toposes}), the Postnikov tower is typically \emph{not} a $0$-connected abstract Postnikov tower: even though all maps induce isomorphisms on $\pi_0$-sheaves, on global sections they typically do not induce bijections on $\pi_0$.
\end{example}
\begin{example}
Let $\V$ be a stable, presentable SM $\infty$-category with a left complete $t$-structure such that the connective part $\V^{\geq 0}$ is closed under finite tensor products. If the mapping spectrum functor $\Map(1_\V, -)\colon \V\lrar \Sp$ sends $\V^{\geq 0}$ to $\Sp^{\geq 0}$, then the abstract Postnikov tower on $\V^{\geq 0}$ from Example \ref{e:postnikov for prestable} is $0$-connected. This is notably the case when $\V=\Mod_R$ with $R$ a connective ring spectrum (with the usual $t$-structure).
\end{example}
\begin{thm}\label{t:main-theorem}
Let $\V$ be a SM $\infty$-category equipped with a multiplicative abstract Postnikov tower $\Phi\colon \V \lrar \Fun(\ET,\V)$. If \(\Phi\) is \(0\)-connected, then the composite
\[\scalebox{0.92}{
\xymatrix@C=1.5pc{
\Cat(\V)\subseteq \Alg_{\Cat}(\V) \ar[r]^-{\Phi_\ast} & \Alg_{\Cat}(\Fun(\ET,\V)) \ar[r] & \Fun(\ET,\Alg_{\Cat}(\V)) \ar[r]^-{(-)^\wedge} & \Fun(\ET,\Cat(\V))
}}\]
defines a multiplicative abstract Postnikov tower $\Phi_{\Cat}$ on $\Cat(\V)$. Furthermore, this abstract Postnikov tower $\Phi_{\Cat}$ is itself $0$-connected.
\end{thm}
Slightly informally (i.e.\ up to Dwyer--Kan equivalence), the abstract Postnikov tower $\Phi_{\Cat}(\CC)$ of a $\V$-enriched $\infty$-category $\CC$ is obtained by applying $\Phi$ to all mapping objects in $\CC$. We will need $\Phi$ to be $0$-connected in order to make use of the cube and tower lemmas (Lemma \ref{l:cube lemma} and \ref{l:tower lemma}).

Note that Theorem \ref{t:main-theorem} can be applied repeatedly:
\begin{defn}\label{d:enriched n-cat}
The $\infty$-category of \textbf{$\V$-enriched $(\infty, n)$-categories}, resp.\ of $n$-categorical algebras, is defined inductively as
$$
\Cat_{n}(\V) := \Cat(\Cat_{n-1}(\V))\qquad\qquad\qquad \Alg_{\Cat_n}(\V)=\Alg_{\Cat}(\Alg_{\Cat_{n-1}}(\V)).
$$
\end{defn}
The fully faithful inclusion $\Cat_n(\V)\subseteq \Alg_{\Cat_n}(\V)$ admits a left adjoint which realizes $\Cat_n(\V)$ as a localization of the $\infty$-category $\Alg_{\Cat_n}(\V)$ at a certain class of `$n$-categorical Dwyer--Kan equivalences'. One then immediately obtains the following:
\begin{cor}\label{c:enriched n-cats}
Let $\V$ be a SM $\infty$-category equipped with a \(0\)-connected multiplicative abstract Postnikov tower $\Phi\colon \V\lrar \Fun(\ET, \V)$. Then there is an induced multiplicative abstract Postnikov tower
$$\xymatrix{
\Phi_{\Cat_n}\colon \Cat_n(\V)\ar[r] & \Tow\big(\Cat_n(\V)\big)
}$$
on the $\infty$-category of $\V$-enriched $(\infty, n)$-categories. Explicitly, $\Phi_{\Cat_n}$ is given (up to $n$-categorical Dwyer--Kan equivalence) by applying $\Phi$ to objects of $n$-morphisms.
\end{cor}
We will now turn to the proof of Theorem \ref{t:main-theorem}. Our strategy will be to first prove a version of Theorem \ref{t:main-theorem} for categorical algebras and then descend to enriched $\infty$-categories. To produce an abstract Postnikov tower at the level of categorical algebras, let $\V$ be a SM $\infty$-category and $\I$ an $\infty$-category. Then there is a natural functor
\begin{equation}\label{e:evaluation}
\I\lrar \Fun^{\otimes}\big(\Fun(\I, \V), \V\big)
\end{equation}
sending each object $i\in \I$ to the symmetric monoidal functor $\Fun(\I, \V)\lrar \V$ evaluating at $i$. Since $\Alg_{\Cat}(\V)$ depends on $\V$ in a symmetric monoidal manner \cite[Corollary 5.7.11]{GH15}, Diagram \ref{e:evaluation} induces a functor
\[\xymatrix@C=1.3pc{
\I\ar[r] &  \Fun^{\otimes}\big(\Fun(\I, \V), \V\big)\ar[rr]^-{\Alg_{\Cat}(-)} & & \Fun^{\otimes}\big(\Alg_{\Cat}(\Fun(\I,\V)), \Alg_{\Cat}(\V)\big)
}\]
or equivalently, by adjunction, a symmetric monoidal functor
\begin{equation}\label{d:categorical algebras in diagrams}
\vphi\colon \Alg_{\Cat}(\Fun(\I, \V))\lrar \Fun(\I, \Alg_{\Cat}(\V)).
\end{equation}
\begin{pro}\label{p:tower-algebra}
Let $\V$ be a SM $\infty$-category and consider the natural symmetric monoidal functor 
$$
\vphi\colon \Alg_{\Cat}(\Fun(\ET, \V))\lrar \Fun(\ET, \Alg_{\Cat}(\V)).
$$
If $\CC$ is a categorical algebra over $\Fun(\ET, \V)$ whose mapping objects belong to the full sub-$\infty$-category $\Tow(\V)$, then $\vphi(\CC)$ is  contained in $\Tow(\Alg_{\Cat}(\V))$.
\end{pro}
\begin{proof}
Let $X$ denote the space of objects of $\CC$, so that for every \(x,y \in X\) we have a mapping object \(\Map_{\CC}(x,y)\in \Tow(\V)\). By construction, the image \(\vphi(\CC)\) sends \(i \in \ET\) to the categorical algebra with object space \(X\) and mapping objects \(\Map_{\CC(i)}(x,y)\). In particular, the composed functor 
\[\xymatrix{
\ET \ar[r]^-{\phi(\CC)} & \Alg_{\Cat}(\V) \ar[r]^-{\Ob} & \cS
}\]
is constant on \(X\), so that we can consider \(\vphi(\CC)\) as a functor from \(\ET\) to the fiber \(\Alg_{\Cat}^X(\V)\). 

We claim that it suffices to show that \(\vphi(\CC)\) is an abstract Postnikov tower when considered as a diagram in \(\Alg_{\Cat}^X(\V)\). To see this, recall that \(\vphi(\CC)\) is an abstract Postnikov tower in $\Alg_{\Cat}(\V)$ if it sends certain designated squares in \(\ET\) to pullback squares in \(\Alg_{\Cat}(\V)\) and a certain \(\ZZ_{\geq 1}^{\triangleleft}\)-diagram to a limit diagram (Remark \ref{r:tow-pullback}). But the fiber inclusion \(\Cat_\V^X \subseteq \Alg_{\Cat}(\V)\) preserves limits indexed by weakly contractible \(\infty\)-categories (such as squares and towers): indeed, this follows from the fact that \(\Ob\colon \Alg_{\Cat}(\V) \lrar \cS\) is a Cartesian fibration such that for every map of spaces $X\lrar Y$, the induced functor $f^*\colon \Alg_{\Cat}^Y(\V)\lrar \Alg_{\Cat}^X(\V)$ between fibers preserves limits  (see, e.g, \cite[Proposition 4.3.1.10]{Lur09}).  

Let us now prove that \(\vphi(\CC)\) is an abstract Postnikov tower in \(\Alg_{\Cat}^X(\V) \simeq \Alg_{\O_{X}}(\V)\), where \(\O_{X}\) is the \(\infty\)-operad which controls the theory of categorical algebras with object space \(X\). 
To this end, note that the symmetry of the Boardman-Vogt tensor product induces a commuting square
$$\xymatrix{
\Fun\big(\ET, \Alg_{\O_{X}}(\V)\big)\ar[r]^\simeq\ar[d] & \Alg_{\O_{X}}(\Fun(\ET, V))\ar[d]\\
\Fun\big(\ET, \Fun(X\times X, \V)\big)\ar[r]_\simeq & \Fun\big(X\times X, \Fun(\ET, \V)\big). 
}$$
The top horizontal equivalence identifies $\vphi(\CC)$ with $\CC$. Since the vertical functors preserve limits and detect equivalences \cite[Corollary 3.2.2.4]{Lur14}, it suffices to show that the image of $\vphi$ under the left vertical functor is an abstract Postnikov tower. But this follows from our assumption that the image of $\CC$ in $\Fun\big(X\times X, \Fun(\ET, \V)\big)$ is contained in $\Fun(X\times X, \Tow(\V))$. 
\end{proof}
\begin{rem}\label{r:towers of cat alg}
The functor of Proposition~\ref{p:tower-algebra} fits into a Cartesian square of symmetric monoidal \(\infty\)-categories
\[\xymatrix{
\Alg_{\Cat}(\Tow(\V)) \hspace{2pt}\ar[d]_-{\Ob}\ar@{^{(}->}[r] & \Tow(\Alg_{\Cat}(\V)) \ar@<-1pt>[d]^{\Tow(\Ob)} \\
\cS \hspace{2pt}\ar@{^{(}->}[r] & \Tow(\cS) \\
}\]
in which the horizontal arrows are full inclusions and the maps induced on vertical fibers are the equivalences of Proposition~\ref{p:quillen-for-algebras} with respect to the various \(\infty\)-operads \(\{\O_X\}_{X \in \cS}\) such that \(\Alg_{\O_X}(\V) \simeq \Alg_{\Cat}^X(\V)\).
\end{rem}
\begin{proof}[Proof of Theorem \ref{t:main-theorem}]
We have to verify that $\Phi_{\Cat}$ is a lax symmetric monoidal section of the (lax) symmetric monoidal functor $\ev_\infty\colon \Fun(\ET, \Cat(\V))\lrar \V$, and that it takes values in the full subcategory $\Tow(\Cat(\V))\subseteq \Fun(\ET, \Cat(\V))$ of abstract Postnikov towers.

For the first assertion, consider the commuting diagram
$$\scalebox{0.92}{\xymatrix@C=1.6pc{
\Cat(\V)\ar@{^{(}->}[r]\ar[d]_= & \Alg_{\Cat}(\V)\ar[d]_{=}\ar[r]^-{\Phi_*} & \Alg_{\Cat}(\Fun(\ET, V))\ar[r]^-{\vphi}\ar[d]_{\ev_\infty} & \Fun(\ET, \Alg_{\Cat}(\V))\ar[d]^{\ev_\infty}\ar[r]^-{(-)^\wedge} & \Fun(\ET, \Cat(\V))\ar[d]^{\ev_\infty}\\
\Cat(\V)\ar@{^{(}->}[r] & \Alg_{\Cat}(\V)\ar[r]_= & \Alg_{\Cat}(\V)\ar[r]_= & \Alg_{\Cat}(\V)\ar[r]_{(-)^\wedge} & \Cat(\V).
}}$$
All arrows are lax symmetric monoidal functors: for the first and last horizontal arrows (in both rows), this follows from \cite[Proposition 5.7.14]{GH15}. The functor $\vphi$ was the symmetric monoidal functor \eqref{d:categorical algebras in diagrams}. The second square commutes since $\Phi_*$ is a multiplicative abstract Postnikov tower and all other squares commute by naturality in the $\infty$-category $\ET$. It will therefore suffice to show that the composed lax symmetric monoidal functor 
\[
\Cat(\V) \subseteq \Alg_{\Cat}(\V) \x{(-)^{\wedge}}{\lrar} \Cat(\V)
\] 
is naturally equivalent to the identity. Indeed, this natural equivalence is given by the counit of the lax symmetric monoidal adjunction between the inclusion \(\Cat(\V) \subseteq \Alg_{\Cat}(\V)\) and the completion functor \cite[Proposition 5.7.14]{GH15}.

For the second assertion, it suffices to verify that for a $\V$-enriched category $\CC$, its image 
\[
\TT^\wedge:= \Phi_{\Cat}(\CC)\colon \ET\lrar \Cat(\V)
\]
defines an abstract Postnikov tower in $\Cat(\V)$. By construction, $\TT^\wedge$ is the objectwise completion of the diagram which applies the multiplicative abstract Postnikov tower $\Phi$ to all mapping objects
\[\TT:=\Phi_\ast(\CC)\colon \ET\lrar \Alg_{\Cat}(\V).
\] 
By Proposition \ref{p:tower-algebra}, $\TT$ is an abstract Postnikov tower in $\Alg_{\Cat}(\V)$. In other words, $\TT$ sends certain designated squares in $\ET$ to pullback squares of categorical algebras and a certain \(\ZZ_{\geq 1}^{\triangleleft}\)-diagram to a limit diagram  (see Remark \ref{r:tow-pullback}). We have to show that these squares remain pullback squares and this limit diagram remains a limit diagram after applying the completion functor objectwise. This will follow from Corollary \ref{c:main-point-cube} and Corollary~\ref{c:main-point-tower} once we verify that $\TT$ sends every arrow in $\ET$ to an isofibration.
  
To this end, consider a map $f\colon \TT(i)\lrar \TT(j)$ of categorical algebras for $i,j\in\E$. By construction, $f$ induces the identity on spaces of objects and for every two objects $x, y\in \TT(i)$, the map
$$
\Map\big(1_{\V}, \Map_{\TT(i)}(x,y)\big)\lrar \Map\big(1_{\V}, \Map_{\TT(j)}(x, y)\big)
$$
induces a bijection on path components since the abstract Postnikov tower $\Phi_*$ is $0$-connected (Definition \ref{d:0-connected}). In particular, it also induces a bijection between the path components consisting of equivalences. Consequently, for every object $x$ in $\TT(i)$ and every equivalence $\alp\colon f(x)\x{\simeq}{\lrar} y$ in $\TT(j)$, there is a (unique up to homotopy) lift of $\alp$ to an equivalence in $\TT(i)$, so that $f$ is an isofibration.

Finally, we have to verify that $\TT^\wedge$ is a $0$-connected abstract Postnikov tower. Note that the monoidal unit of $\Cat(\V)$ is the completion $[0]_{1_\V}^\wedge$ of the one object enriched $\infty$-category with endomorphisms $1_\V$. In particular, the functor 
$$
\Map_{\Cat(\V)}(1_{\Cat(\V)}, -)\colon \Cat(\V)\lrar \cS
$$
is equivalent to the functor sending a $\V$-enriched $\infty$-category to its space of objects. To see that $\TT^\wedge$ is $0$-connected, it therefore suffices to verify that the diagram of $1$-truncated object spaces
$$\xymatrix{
\tau_{\leq 1}\Ob(\TT^\wedge)\colon \ET\ar[r]^-{\TT^{\wedge}} & \Cat(\V)\ar[r]^-{\Ob}& \cS\ar[r]^-{\tau_{\leq 1}} & \cS_{\leq 1}
}$$
is essentially constant. By \cite[Theorem 5.6.2]{GH15} and the fact that truncation preserves colimits, this diagram of truncated spaces of objects arises as the geometric realization of the $\ET$-diagram of Segal groupoid objects
$$\xymatrix{
\tau_{\leq 1}\Ob(\TT^\wedge) & \ar[l] \tau_{\leq 1}\Ob(\TT) & \ar@<0.5ex>[l]\ar@<-0.5ex>[l]\tau_{\leq 1}\Map\big(J, \TT\big) & \dots\ar@<1ex>[l]\ar@<-1ex>[l]\ar[l]
}$$
where the mapping space is taken in the $\infty$-category of categorical algebras. Note that the tower of categorical algebras $\TT$ is essentially constant at the level of objects. Furthermore, the fibers of the map $(d_0, d_1)\colon \tau_{\leq 1}\Map_{\Alg_{\Cat}(\V)}\big(J, \TT\big)\lrar \tau_{\leq 1}\Ob(\TT)\times \tau_{\leq 1}\Ob(\TT)$ are exactly the sets of homotopy classes of equivalences in $\TT$; we have already seen above that these are essentially constant in $\ET$ as well. It therefore follows that the entire $\ET$-diagram of Segal groupoid objects is essentially constant, hence its colimit is essentially constant as well.
\end{proof}

\section{Local systems on $(\infty, n)$-categories}\label{s:local systems}
In this section, we spell out the contents of Theorem \ref{t:main-theorem} and Corollary \ref{c:enriched n-cats} in the setting of $(\infty, n)$-categories. There are many equivalent models for the $\infty$-category of $(\infty, n)$-categories, one of which is the $\infty$-category $\Cat_n(\cS)$ of $\cS$-enriched $(\infty, n)$-categories of Definition \ref{d:enriched n-cat} \cite[Corollary 7.21]{Hau15}. This model is particularly well-adapted to definitions that proceed by induction on mapping objects, such as the following:
\begin{defn}[{\cite[Definition 6.1.1]{GH15}}]
An $(m, 0)$-category is defined to be an $m$-truncated space. For any $0\leq n\leq m$, an $(\infty, n)$-category $\C$ is called an \textbf{$(m, n)$-category} if each mapping $(\infty, n-1)$-category is an $(m-1, n-1)$-category.
\end{defn}
The inclusion $\Cat_{(m, n)}\subseteq \Cat_{(\infty, n)}$ of the full subcategory of $(m, n)$-categories admits a left adjoint  \cite[Proposition 6.1.7]{GH15}, sending an $(\infty, n)$-category $\C$ to its \textbf{homotopy $(m, n)$-category} $\Ho_{(m, n)}\C$. The universal properties of the homotopy $(m, n)$-categories imply that they fit into a converging tower
\begin{equation}\label{d:homotopy categories}\vcenter{\xymatrix{
\C\ar[r] & \dots\ar[r] & \Ho_{(m, n)}(\C)\ar[r] & \Ho_{(m-1, n)}(\C)\ar[r] & \dots\ar[r] & \Ho_{(n+1, n)}(\C).
}}\end{equation}
Theorem \ref{t:main-theorem} and Corollary \ref{c:enriched n-cats} then yield the following more precise statement of Theorem \ref{t:main-theorem-intro}:
\begin{thm}\label{t:main-theorem-oo-n-cats}
For any $(\infty, n)$-category $\C$, the tower of homotopy $(n, m)$-categories \eqref{d:homotopy categories} is part of a (functorial, multiplicative) abstract Postnikov tower: for each $a\geq 2$, there exists a parametrized spectrum object
$$
\rH\pi_a(\C)\in \T_{\Ho_{(n+1, n)}(\C)}\big(\Cat_{(\infty, n)}\big)
$$
and a pullback square of $(\infty, n)$-categories
$$\xymatrix{
\Ho_{(n+a, n)}\C\ar[d]\ar[r] & \Ho_{(n+1, n)}\C\ar[d]^0\\
\Ho_{(n+a-1, n)}\C\ar[r]_-{k_a} & \Omega^{\infty}\big(\Sigma^{a+1}\rH\pi_a(\C)\big).
}$$
\end{thm}
\begin{proof}
Observe that (by induction) the full subcategory $\Cat_{(m, n)}\subseteq \Cat_{(\infty, n)}$ can be identified with the $\infty$-category $\Cat_n(\cS_{\leq m-n})$ of $(\infty, n)$-categories enriched in $(m-n)$-truncated spaces. Consequently, the localization
$$\xymatrix@C=5pc{
\Ho_{(m, n)}\colon \Cat_{(\infty, n)} = \Cat_n(\cS)\ar@<1ex>[r]^-{\Cat_n(\tau_{\leq m-n})} & \hspace{2pt} \Cat_n(\cS_{\leq m-n})=\Cat_{(m, n)}\ar@<1ex>@{_{(}->}[l]^{\Cat_n(\iota)}\colon \iota
}$$
is induced by the monoidal localization of spaces $\tau_{\leq m-n}\colon \cS\adj \cS_{\leq m-n}\colon \iota$, using the functoriality of $\V$-enriched $(\infty, n)$-categories in lax monoidal functors (cf.\ \cite[Corollary 5.7.6]{GH15}). 

Now let $\Phi\colon \cS\lrar \Fun(\ET, \cS)$ be the multiplicative, $0$-connected abstract Postnikov tower on spaces refining the classical Postnikov tower (Example \ref{e:postnikov for spaces}). By Theorem \ref{t:main-theorem} and Corollary \ref{c:enriched n-cats}, this induces a multiplicative abstract Postnikov tower $\Phi_{\Cat_n}$ on $\Cat_{(\infty, n)}$. Note that we can view $\Phi_{\Cat_n}$ as a diagram $\ET\lrar \Fun^{\otimes, \mathrm{lax}}(\Cat_{(\infty, n)}, \Cat_{(\infty, n)})$ of (lax symmetric monoidal) endofunctors of $\Cat_{(\infty, n)}$. Forgetting about the $k$-invariants, the underlying tower of $\Phi_{\Cat_n}$ is given by the tower of functors
$$\xymatrix{
\id\ar[r] & \dots \ar[r] & \Cat_n(\tau_{\leq a})\ar[r] & \Cat_n(\tau_{\leq a-1})\ar[r] & \dots\ar[r] & \Cat_{n}(\tau_{\leq 1}). 
}$$
When evaluated on $\C$, the above argument identifies this with the tower of homotopy $(m, n)$-categories \eqref{d:homotopy categories}. Furthermore, by construction $\Phi_{\Cat_n}$ provides $k$-invariants with coefficients in certain parametrized spectra over $\Ho_{(n+1, n)}(\C)$. We simply \emph{define} the spectra $\rH\pi_a(\C)$ to be the $(a+1)$-fold desuspensions of these spectra.
\end{proof}
The proof of Theorem \ref{t:main-theorem-oo-n-cats} is not completely satisfying because the parametrized spectra $\rH\pi_a(\C)$ are defined somewhat implicitly. In the remainder of this section, we will explain how (as the notation suggests) the parametrized spectra $\rH\pi_a(\C)$ can be considered as the Eilenberg--Maclane spectra associated to \textbf{local systems of abelian groups} on $\Ho_{(n+1, n)}(\C)$, as considered in \cite{Lur09b}.

\subsection{$t$-orientations on tangent categories}
\begin{defn}\label{d:t-structure}
Let $\V$ be an $\infty$-category with finite limits and let $\pi\colon \T\V\lrar \V$ be its tangent bundle. A \textbf{$t$-orientation} on $\T\V$ is a tuple of full subcategories $\big(\T^{\geq 0}\V, \T^{\leq 0}\V\big)$ such that:
\begin{enumerate}
\item For each $\pi$-Cartesian arrow $E\lrar F$ in $\T\V$ with $F\in \T^{\leq 0}\V$, we have that $E\in \T^{\leq 0}\V$.
\item For every $X\in \V$, the tuple
$$
\Big(\T^{\geq 0}\V\cap \T_{X}\V, \T^{\leq 0}\V\cap \T_{X}\V\Big)
$$
defines a $t$-structure on the stable $\infty$-category $\T_X\V$.
\end{enumerate}
In this case, we will refer to $\T^{\heartsuit}\V=\T^{\geq 0}\V\cap \T^{\leq 0}\V$ as the \emph{heart} of the $t$-orientation. 
\end{defn}
Note that Condition (1) of Definition \ref{d:t-structure} is equivalent to $\T^{\leq 0}\V\lrar \V$ being a Cartesian fibration and the inclusion $\T^{\leq 0}\V\lrar \T\V$ preserving Cartesian edges. 
\begin{lem}\label{l:heart}
Let $\V$ be an $\infty$-category with finite limits and let $(\T^{\geq 0}\V, \T^{\leq 0}\V)$ be a $t$-orientation on its tangent bundle. Then:
\begin{enumerate}
\item the restriction of the projection $\pi\colon \T\V\lrar \V$ to each of the three subcategories $\T^{\geq 0}, \T^{\leq 0}\V$ and $\T^{\heartsuit}\V$ is a Cartesian fibration.
\item there exists a commuting square of adjunctions over $\V$, i.e.\ in $\Cat_{\infty}/\V$, of the form
$$\xymatrix@C=3.8pc@R=3pc{
\parbox[c][1em][t]{2em}{$\T^{\heartsuit}\V$}\ar@<1ex>@{^{(}->}[r]\ar@<1ex>@{^{(}->}[d] & \parbox[c][1em][c]{2.2em}{$\T^{\leq 0}\V$}\ar@<1ex>[l]^-{\tau_{\geq 0}}_-{\perp}\ar@<1ex>@{^{(}->}[d]\\
\parbox[c][1em][c]{2.2em}{$\T^{\geq 0}\V$}\ar@<1ex>@{^{(}->}[r]\ar@<1ex>[u]_-{\dashv}^-{\tau_{\leq 0}} & \parbox[c][1em][c]{2em}{$\T\V.$}\ar@<1ex>[l]^-{\tau_{\geq 0}}_-{\perp}\ar@<1ex>[u]_-{\dashv}^-{\tau_{\leq 0}}
}$$
Furthermore, all right adjoint functors preserve Cartesian edges.
\end{enumerate}
\end{lem}
In particular, note that $\T^{\heartsuit}\V\lrar \V$ is a Cartesian fibration whose fibers are (ordinary) abelian categories.
\begin{proof}
For each $X\in \V$, the fiber $\T_X\V$ comes equipped with a $t$-structure. In particular, for each $X$ there are coreflective localizations \cite[Proposition 1.2.1.5]{Lur14}
\begin{equation}\label{eq:connective and heart as localization}
\xymatrix@C=3pc{
\T_X^{\geq 0}\V\hspace{1pt}\ar@<1ex>@{^{(}->}[r] & T_X\V\ar@<1ex>[l]^-{\tau_{\geq 0}}_-{\perp} & & 
\T_X^{\heartsuit}\V\hspace{1pt}\ar@<1ex>@{^{(}->}[r] & T_X^{\leq 0}\V.\ar@<1ex>[l]^-{\tau_{\geq 0}}_-{\perp}
}\end{equation}
The functors $\tau_{\geq 0}$ realize their codomain as the localization of the domain at the $(-1)$-coconnective morphisms, i.e.\ those morphisms whose cofiber in $\T_X\V$ is contained in $\T^{\leq 0}_X\V$. By Condition (1), each morphisms $f\colon X\lrar Y$ in $\V$ induces a left $t$-exact functor $f^*\colon \T_Y\V\lrar \T_X\V$ between the fibers. It follows that the $(-1)$-coconnective morphisms in $\T\V$ and $\T^{\leq 0}\V$ are stable under the functors $f^*$. Let $\U^{\geq 0}$ (resp.\ $\U^{\heartsuit}$) denote the $\infty$-category obtained from $\T\V$ (resp.\ $\T^{\leq 0}\V$) by localizing at the $(-1)$-coconnective arrows in each fiber. By \cite[Proposition 2.1.4]{Hin16}, this yields maps of Cartesian fibrations (preserving Cartesian arrows)
\begin{equation}\label{eq:fiberwise connective cover}\vcenter{\xymatrix{
\T\V\ar[rr]^-{\tau_{\geq 0}}\ar[rd]_\pi & & \U^{\geq 0}\ar[ld] & & \T^{\leq 0}\V\ar[rd]_\pi\ar[rr]^-{\tau_{\geq 0}} & & \U^{\heartsuit}\ar[ld]\\
& \V & & & & \V.
}}\end{equation}
On the fiber over an object $X\in \V$, these maps can be identified with the localization functors from \eqref{eq:connective and heart as localization}. In particular, it follows from \cite[Proposition 7.3.2.6]{Lur14} that the localizations from \eqref{eq:fiberwise connective cover} both admit a left adjoint over $\V$. These left adjoints are (fiberwise) fully faithful and identify $\U^{\geq 0}$ and $\U^{\heartsuit}$ with the full subcategories $\T^{\geq 0}\V$ and $\T^{\heartsuit}\V$, respectively. 

In particular, this shows that the projections from $\T^{\geq 0}\V$ and $\T^{\heartsuit}\V$ to $\V$ are Cartesian fibrations, proving (1). Furthermore, the functors from \eqref{eq:fiberwise connective cover} provide the horizontal right adjoints (relative to $\V$) in (2). Finally, the inclusions $\T^{\heartsuit}\V\lrar \T^{\geq 0}\V$ and $\T^{\leq 0}\V\lrar \T\V$ admit left adjoints over $\V$ by \cite[Proposition 7.3.2.6]{Lur14}.
\end{proof}
\begin{define}
Let $\V$ be an SM $\infty$-category with finite limits. A $t$-orientation on $\T\V$ is \textbf{monoidal} if $\T^{\geq 0}\V$ is closed under the square-zero tensor product and contains the unit.
\end{define}
\begin{example}\label{e:postnikov for prestable on TS}
Consider the full subcategories of excisive functors $F\colon \cS^{\fin}_\ast\lrar \cS$
$$
\T^{\geq 0}\cS\subseteq \T\cS \qquad\qquad\qquad \T^{\leq 0}\cS\subseteq \T\cS
$$
such that for every $n$, the map $F(S^n)\lrar  F(\ast)$ has $n$-connected, resp.\ $n$-truncated fibers. This defines a $t$-orientation on $\T\cS$, whose restriction to each fiber $\T_X\cS\simeq \Fun(X, \Sp)$ consists of diagrams of connective, resp.\ coconnective spectra. Furthermore, this $t$-orientation is monoidal (the square zero monoidal structure simply being the Cartesian product).

In particular, the heart $\T^\heartsuit\cS$ can be identified with the $\infty$-category of \textbf{local systems of abelian groups}. The inclusion $\T^\heartsuit\cS\subseteq \T\cS$ sends a local system of abelian groups $\A$ to the corresponding parametrized Eilenberg--Maclane spectrum $\rH\A$.
\end{example}
\begin{example}\label{e:t-structure additive}
Let $\V$ be an additive presentable $\infty$-category, so that $\T\V\simeq \V\times \Sp(\V)$ (Example \ref{e:stable}). Then any $t$-structure on $\Sp(\V)$ determines a $t$-orientation on $\V$.
Now suppose that $\V$ is furthermore symmetric monoidal and recall that the square zero monoidal structure on $\T\V$ can be identified with the monoidal structure on $\V\times \Sp(\V)$ given by Remark \ref{r:square zero in additive setting}. From this description, one sees that a $t$-structure on $\Sp(\V)$ determines a \emph{monoidal} $t$-orientation on $\T\V$ if and only if $\Sp(\V)^{\geq 0}$ is closed under taking the tensor product in $\Sp(\V)$ with objects of the form $\Sigma^\infty(X)$, for $X\in \V$.

In particular, one can always construct a monoidal $t$-orientation on $\T\V$ using the $t$-structure on $\Sp(\V)$ whose connective part is the smallest full subcategory which is closed under colimits and contains all $\Sigma^\infty(X)$ with $X\in \V$. When $\V\simeq \Sp(\V)$ is stable, this simply produces the trivial $t$-structure and if $\V$ is prestable, this produces the canonical $t$-structure on $\Sp(\V)$ whose connective part is $\V$ itself. 
\end{example}
Let $\V$ be a SM $\infty$-category with finite limits and suppose that $\T\V$ carries a monoidal $t$-orientation. If $\O$ is an $\infty$-operad, we can use Proposition \ref{p:quillen-for-algebras} to identify the Cartesian fibration $\pi\colon \T\Alg_{\O}(\V)\lrar\Alg_{\O}(\V)$ with $\Alg_{\O}(\T\V)\lrar \Alg_{\O}(\V)$. Using this identification, consider the two full subcategories
$$
\T^{\geq 0}\Alg_{\O}(\V)=\Alg_{\O}(\T^{\geq 0}\V) \qquad\qquad \qquad \T^{\leq 0}\Alg_{\O}(\V)=\Alg_{\O}(\T^{\leq 0}\V)
$$
where we view $\T^{\geq 0}\V$ and $\T^{\leq 0}\V$ as full suboperads of $\T\V$. In other words, these are the full subcategories of $\O$-algebras in $\T\V$ whose underlying objects (for every colour $x\in \O$) are $0$-connective, resp.\ $0$-coconnective in $\T\V$.
\begin{pro}\label{p:t-structure on tangent of algebras}
These two full subcategories $\T^{\geq 0}\Alg_{\O}(\V)$ and $\T^{\leq 0}\Alg_{\O}(\V)$ define a monoidal $t$-orientiation on $\T\Alg_{\O}(\V)$. For every colour $x\in \O$, the forgetful functor $x^*\colon \T\Alg_\O(\V)\lrar \T\V$ is $t$-exact, i.e.\ it preserves both $0$-connective and $0$-coconnective objects.
\end{pro}
\begin{proof}
To see that the $t$-orientation is monoidal, note that the full subcategory $\T^{\geq 0}\Alg_{\O}(\V)\simeq \Alg_{\O}(\T^{\geq 0}\V)\subseteq \Alg_{\O}(\T\V)$ is closed under tensor products, since a symmetric monoidal functor induces a symmetric monoidal functor on $\O$-algebras.

To verify condition (1) of Definition \ref{d:t-structure}, notice that a morphism in $\Alg_{\O}(\T\V)$ is $\pi$-Cartesian if and only if for every colour $x\in \O$, its image under $x^*\colon \Alg_{\O}(\T\V)\lrar \T\V$ is a Cartesian arrow. This immediately implies that for every Cartesian arrow in $\Alg_{\O}(\T\V)$ whose codomain is contained in $\T^{\leq 0}\Alg_{\O}(\V)$, the domain is contained in $\T^{\leq 0}\Alg_{\O}(\T\V)$ as well.

For condition (2), note that Lemma \ref{l:heart} provides a colocalization $\T^{\geq 0}\V\adj \T\V$ whose counit maps to an equivalence in $\V$. Since the inclusion $\T^{\geq 0}\V\lrar \T\V$ is symmetric monoidal, its right adjoint $\tau_{\geq 0}$ is lax symmetric monoidal. We therefore obtain a colocalization at the level of $\O$-algebras, whose counit maps to an equivalence in $\Alg_{\O}(\V)$. Furthermore, this commutes with the forgetful functors for each colour $x\in\O$
$$\xymatrix{
\T^{\geq 0}\Alg_{\O}(\V)\simeq \Alg_{\O}(\T^{\geq 0}\V)\ar@<1ex>@{^{(}->}[r]\ar[d]_{x^*} & \Alg_{\O}(\T\V)\ar@<1ex>[l]_-{\perp}\ar[d]^{x^*}\\
\T^{\geq 0}\V\ar@<1ex>@{^{(}->}[r] & \T\V.\ar@<1ex>[l]
}$$
On fibers over an $\O$-algebra $A$, we therefore obtain commuting colocalizations
$$\xymatrix{
\T^{\geq 0}_A \Alg_{\O}(\V)\ar@<1ex>@{^{(}->}[r]\ar[d]_{x^*} & \T_A\Alg_{\O}(\V)\ar@<1ex>[l]_-{\perp}^{\tau_{\geq 0}}\ar[d]^{x^*}\\
\T^{\geq 0}_{x^*A}\V\ar@<1ex>@{^{(}->}[r] & \T_{x^*A}\V.\ar@<1ex>[l]^{\tau_\geq 0}
}$$
In particular, it follows that an object $X\in \T_A\Alg_{\O}(\V)$ is:
\begin{itemize}
\item contained in $\T_A^{\geq 0}\Alg_{\O}(\V)$ if and only if $\tau_{\geq 0}(X)\simeq X$.
\item contained in $\T_A^{\leq -1}\Alg_{\O}(\V)$ if and only if for every colour $x\in \O$, $x^*X\in \T_{x^*A}\V$ is $(-1)$-coconnective, i.e.\ $\tau_{\geq 0}(x^*X)\simeq 0$. In turn, this is equivalent to $\tau_{\geq 0}(X)\simeq 0$ in $\T_A\Alg_{\O}(\V)$.
\end{itemize}
By \cite[Proposition 1.2.1.16]{Lur14}, the subcategories $\T_A^{\geq 0}\Alg_{\O}(\V)$ and $\T_A^{\leq 0}\Alg_{\O}(\V)$ then determine a $t$-structure on $\T_A\Alg_{\O}(\V)$ if and only if $\T_A^{\geq 0}\Alg_{\O}(\V)$ is closed under extensions. This holds because the forgetful functors $x^*$ preserve extensions and $\T_{x^*A}^{\geq 0}\V$ is closed under extensions.
\end{proof}
\begin{example}
Suppose that $\V$ is a prestable SM $\infty$-category and \(\O\) an \(\infty\)-operad. Endowing $\T\V \simeq \V \times \Sp(\V)$ with its ``constant'' monoidal $t$-orientation of Example \ref{e:t-structure additive} and applying Proposition \ref{p:t-structure on tangent of algebras} we obtain a $t$-orientation on $\T\Alg_{\O}(\V)$, and hence a $t$-structure on $\T_A\Alg_{\O}(\V)$ for any $\O$-algebra $A$. Under the identification % by . For every $\O$-algebra $A$ in $\V$, the resulting $t$-structure on 
$\T_A\Alg_{\O}(\V)\simeq \Mod_A\big(\Sp(\V)\big)$ (see Example \ref{e:square zero extensions in stable}) %. This 
this is simply the $t$-structure whose connective part is given by $\Mod_A(\V)\subseteq \Mod_A\big(\Sp(\V)\big)$.
\end{example}

\subsection{Tangent bundle of enriched $\infty$-categories}
To prove a version of Proposition \ref{p:t-structure on tangent of algebras} for enriched $\infty$-categories, we will need a description of the tangent bundle to enriched $\infty$-categories along the lines of Proposition \ref{p:tangent to algebras}:
\begin{pro}\label{p:tangent to categories}
Let $\V$ be a differentiable presentable SM $\infty$-category such that $1_\V$ is compact. Then there exists a natural equivalence of SM $\infty$-categories
\begin{equation}\label{d:L from previous paper}\vcenter{\xymatrix{
\Cat(\T\V)\ar[rd]_-{\Cat(p)}\ar[rr]^-{\L}_-\sim &  & \T\Cat(\V)\ar[ld]^-\pi\\
& \Cat(\V)
}}\end{equation}
where $\Cat(p)$ is induced by the (monoidal) projection $p\colon \T\V\lrar \V$ and $\pi$ is the natural projection.
\end{pro}
\begin{rem}\label{r:differentiable}
Recall from \cite[Definition 6.1.1.6]{Lur14} that a presentable $\infty$-category $\V$ is differentiable if the sequential colimit functor $\colim\colon \Fun(\NN, \V)\lrar \V$ is left exact. In particular, any compactly generated $\infty$-category is differentiable.

If a presentable $\infty$-category $\V$ is differentiable, then $\Cat(\V)$ is differentiable as well. To see this, pick a monoidal model category $\mathbf{V}$ presenting $\V$ in which all objects are cofibrant \cite[Remark 4.1.8.9]{Lur14}, so that $\Cat(\V)$ arises from the model category $\Cat(\mathbf{V})$ on $\mathbf{V}$-enriched categories \cite{Hau15}. It then follows from \cite[Corollary 3.1.12]{part2} that $\Cat(\V)$ is again differentiable. 
Furthermore, the monoidal unit $1_{\Cat(\V)}$ is compact. Indeed, the corepresentable functor $\Map(1_{\Cat(\V)}, -)$ can be identified with the functor taking spaces of objects. This functor can be factored as
$$\xymatrix{
\Cat(\V)\ar[r] & \Cat(\cS)\simeq \Cat_\infty\ar[r]^-{\mathrm{Core}} & \cS
}$$
where the first functor preserves filtered colimits because $1_\V$ was compact and the second functor, taking the core (or maximal sub-$\infty$-groupoid), preserves filtered colimits as well. It follows that Proposition \ref{p:tangent to categories} can be applied inductively.
\end{rem}
\begin{proof}
The idea will be to first work at the level of categorical algebras and then localize at the Dwyer--Kan equivalences to obtain the desired triangle on enriched $\infty$-categories. We then prove that $\L$ induces an equivalence on the fibers over a fixed object in $\Cat(\V)$ by using the results of \cite{part2}. 

\medskip 

\noindent \textbf{Step 1.} Consider the natural symmetric monoidal functor $\vphi\colon \Alg_{\Cat}(\Fun(\cS^{\fin}_\ast, \V))\lrar \Fun(\cS^{\fin}_{\ast}, \Alg_{\Cat}(\V))$ from \eqref{d:categorical algebras in diagrams}, where $\Fun(\cS^{\fin}_\ast, \V)$ comes equipped with the pointwise tensor product from $\V$. By Remark \ref{r:towers of cat alg}, this is a fully faithful inclusion, whose essential image consists of diagrams of categorical algebras whose underlying diagram of spaces of objects is essentially constant. In particular, note that $\vphi$ sends Dwyer--Kan equivalences between categorical algebras in $\Fun(\cS^{\fin}_\ast, \V)$ to $\cS^{\fin}_\ast$-diagrams of Dwyer--Kan equivalences in $\Alg_{\Cat}(\V)$.

The proof of Proposition \ref{p:tower-algebra} shows that this restricts to a functor $\vphi\colon \Alg_{\Cat}(\T\V)\lrar \T(\Alg_{\Cat}(\V))$, which fits into a pullback square of symmetric monoidal functors (see Remark \ref{r:towers of cat alg})
$$\xymatrix{
\Alg_{\Cat}(\T\V)\hspace{2pt}\ar@{^{(}->}[r]^-\vphi\ar@<-1pt>[d] & \T\Alg_{\Cat}(\V)\ar[d]\\
\cS\hspace{2pt}\ar@{^{(}->}[r] & \T\cS.
}$$
Note that $\vphi$ commutes with evaluation at $\ast\in \cS^{\fin}_{\ast}$, so that it fits into a commuting triangle of symmetric monoidal functors
$$\xymatrix{
\Alg_{\Cat}(\T\V)\hspace{2pt}\ar@{^{(}->}[rr]^\vphi \ar[rd]_{\Alg_{\Cat}(p)} & & \T\Alg_{\Cat}(\V)\ar[ld]^\pi\\
& \Alg_{\Cat}(\V) & 
}$$
where $\Alg_{\Cat}(p)$ is induced by $p\colon \T\V\lrar \V$ and $\pi$ is the natural projection. By functoriality of tangent $\infty$-categories, the functor $\Ob\colon \Alg_{\Cat}(\V)\lrar \cS$  induces a functor between tangent bundles $\T\Alg_{\Cat}(\V)\lrar \T\cS$ which preserves Cartesian arrows. Consequently, given a $\pi$-Cartesian arrow in $\T\Alg_{\Cat}(\V)$ whose codomain is contained in $\Alg_{\Cat}(\T\V)$, i.e.\ the underlying parametrized spectrum of objects is constant, then its domain is contained in $\Alg_{\Cat}(\T\V)$ as well. In particular, $\Alg_{\Cat}(p)$ is a Cartesian fibration and $\vphi$ preserves Cartesian arrows.

\medskip

\noindent\textbf{Step 2.} To obtain the desired functor $\L$ as in \eqref{d:L from previous paper}, we need to invert the Dwyer--Kan equivalences. To this end, note that localization at the Dwyer--Kan equivalence determines a commuting diagram on tangent $\infty$-categories
$$\xymatrix{
\T\Alg_{\Cat}(\V)\ar@<1ex>[r]\ar[d] & \T\Cat(\V)\ar[d]\ar@<1ex>@{_{(}->}[l]_-{\perp}\\
\Alg_{\Cat}(\V)\ar@<1ex>[r] & \Cat(\V)\ar@<1ex>@{_{(}->}[l]_-{\perp}.
}$$
Let us say that a map in $\T\Alg_{\Cat}(\V)$ is a Dwyer--Kan equivalence if it becomes an equivalence in $\T\Cat(\V)$. Then $\pi$ preserves Dwyer--Kan equivalences. Furthermore, $\vphi$ sends a Dwyer--Kan equivalence in $\Alg_{\Cat}(\T\V)\subseteq \Alg_{\Cat}\big(\Fun(\cS^{\fin}_\ast, \V)\big)$ to a pointwise Dwyer--Kan equivalence in $\T\Alg_{\Cat}(\V)\subseteq \Fun\big(\cS^{\fin}_\ast, \Alg_{\Cat}(\V)\big)$. Such pointwise Dwyer--Kan equivalences certainly become equivalences in $\T\Cat(\V)$. We can therefore inverting the Dwyer--Kan equivalences, and obtain the desired triangle of symmetric monoidal functors \eqref{d:L from previous paper} \cite[Proposition 4.1.7.4]{Lur14}. 

To verify that the functor $\L$ is an equivalence, we will first check that $\pi$ and $\Cat(p)$ are Cartesian fibrations and that $\L$ preserves Cartesian arrows. Assuming this, it will then suffice to prove that $\L$ induces a fiberwise equivalence \cite[Corollary 2.4.4.4]{Lur09} (see Step 3).

Note that $\pi$ is clearly a Cartesian fibration, being the projection of the tangent bundle to $\Cat(\V)$. On the other hand, to see that $\Cat(p)$ is a Cartesian fibration, let us make the following two observations about Dwyer--Kan equivalences in $\Alg_{\Cat}(\T\V)$:
\begin{enumerate}\setlength{\itemsep}{4pt}
\item A map $f\colon \CC\lrar \DD$ in $\Alg_{\Cat}(\T\V)$ is fully faithful if and only if it is an $\Alg_{\Cat}(p)$-Cartesian lift of a fully faithful map in $\Alg_{\Cat}(\V)$. Indeed, $f$ is fully faithful if and only if each $f_*\colon \Map_{\CC}(x, y)\lrar \Map_{\DD}(x, y)$ is an equivalence in $\T\V$. In turn, this is equivalent to each $f_*$ being $p$-Cartesian and mapping to an equivalence in $\V$. This means precisely that $f$ is an $\Alg_{\Cat}(p)$-Cartesian arrow mapping to a fully faithful map in $\Alg_{\Cat}(\V)$.

\item $\Alg_{\Cat}(p)$ detects essentially surjective functors. Indeed, note that the monoidal unit of the square zero monoidal structure on $\T\V$ is equivalent to $0_p(1_\V)$, where $0_p$ denotes the left adjoint section of $p\colon \T\V\lrar \V$. Consequently, there is a natural equivalence of lax monoidal functors $\Map(1_{\T\V}, -)\simeq \Map(1_{\V}, p(-))$. It follows that the $\cS$-enriched categorical algebra underlying $\CC\in \Alg_{\Cat}(\T\V)$ is naturally equivalent to the $\cS$-enriched categorical algebra underlying the image of $\CC$ in $\Alg_{\Cat}(\V)$. Since being essentially surjective is determined on the underlying $\cS$-enriched categorical algebras, the claim follows.
\end{enumerate}
Combining (1) and (2), one sees that a map in $\Alg_{\Cat}(\T\V)$ is a Dwyer--Kan equivalence if and only if it is an $\Alg_{\Cat}(p)$-Cartesian arrow covering a Dwyer--Kan equivalence in $\Alg_{\Cat}(\V)$. It then follows from \cite[Proposition 2.1.4]{Hin16} that the induced functor between localizations $\Cat(p)\colon \Cat(\T\V)\lrar \Cat(\V)$ is a Cartesian fibration, which fits into a pullback square
\begin{equation}\label{d:localization diagram}\vcenter{\xymatrix{
\Alg_{\Cat}(\T\V)\ar[r]^-{(-)^\wedge}\ar[d]_{\Alg_{\Cat}(p)} & \Cat(\T\V)\ar[d]^{\Cat(p)}\\
\Alg_{\Cat}(\V)\ar[r]_-{(-)^\wedge} & \Cat(\V).
}}\end{equation}
In other words, for every categorical algebra $\CC$, the fiber of $\Cat(p)$ over $\CC^\wedge$ is equivalent to the fiber of $\Alg_{\Cat}(p)$ over $\CC$.

It remains to verify that $\L$ preserves Cartesian edges. Unraveling the definitions, the functor $\L$ arises as the top horizontal composite in a diagram
$$\scalebox{0.95}{\xymatrix@C=1.6pc{
\Cat(\T\V)\ar@{^{(}->}[r]\ar[d] & \Alg_{\Cat}(\T\V)\ar@{^{(}->}[r]^-{\vphi}\ar[d] & \Fun(\cS^{\fin}_{\ast}, \Alg_{\Cat}(\V)\big)\ar[d]\ar[r]^-{(-)^\wedge} & \Fun(\cS_{\ast}^{\fin}, \Cat(\V))\ar[d]\ar[r]^-{(-)^{\exc}} & \T\Cat(\V)\ar[d]\\
\Cat(\V)\ar@{^{(}->}[r] & \Alg_{\Cat}(\V)\ar[r]_= & \Alg_{\Cat}(\V)\ar[r]_-{(-)^\wedge} & \Cat(\V)\ar[r]_= & \Cat(\V).
}}$$
Here $(-)^\wedge$ takes the completion (pointwise in $\cS^{\fin}_\ast$) and $(-)^{\exc}$ takes the excisive approximation (Definition \ref{d:aff}). Note that the bottom composition is equivalent to the identity and that all vertical functors are Cartesian fibrations.

The first two functors $\Cat(\T\V)\lrar \Alg_{\Cat}(\T\V)$ and $\vphi$ preserve Cartesian arrows (by Remark \ref{r:excisive approximation preserves cartesian arrows} and Step 1) and the last functor $(-)^{\exc}$ preserves Cartesian arrows because $\Cat(\V)$ is differentiable (Remark \ref{r:excisive approximation preserves cartesian arrows} and \ref{r:differentiable}). It therefore remains to verify that $(-)^\wedge$ preserves Cartesian arrows contained in the essential image of $\vphi$. To see this, let $f\colon \CC\lrar \DD$ be a map of categorical algebras in $\V$ and let $F\colon \cS^{\fin}_\ast\lrar \Alg_{\Cat}(\V)$ be a diagram in the image of $\vphi$ such that $F(*)\simeq \DD$. Then a Cartesian lift of $f$ is given by the canonical map from the fiber product
$$
F\times_{\Delta(\DD)} \Delta(\CC)\lrar F
$$
where $\Delta(-)$ takes the constant $\cS^{\fin}_\ast$-diagram. To see that $(-)^{\wedge}$ sends this to a Cartesian arrow in $\Fun(\cS^{\fin}_\ast, \Cat(\V))$, it suffices to verify that the natural transformation
$$
F^\wedge\times_{\Delta(\DD^{\wedge})} \Delta(\CC^\wedge)\lrar \big(F\times_{\Delta(\DD)} \Delta(\CC)\big)^\wedge
$$
is an equivalence, since Cartesian arrows in $\Fun(\cS^{\fin}_\ast, \Cat(\V))$ are also given by the canonical map from the fiber product. When evaluated at a finite pointed space $S$, this map can be identified with the map in $\Cat(\V)$
\begin{equation}\label{d:completion of cartesian}
F(S)^\wedge\times_{F(\ast)^\wedge} \CC^\wedge\lrar \big(F(S)\times_{F(\ast)} \CC\big)^\wedge. 
\end{equation}
But now notice that the map $F(S)\lrar F(\ast)$ is an equivalence on spaces of objects (since $F$ was in the image of $\vphi$) and admits a section $F(\ast)\lrar F(S)$  induced by the basepoint $\ast\lrar S$. This implies that $F(S)\lrar F(\ast)$ is an isofibration, so that Corollary \ref{c:main-point-cube} implies that the map \eqref{d:completion of cartesian} is indeed an equivalence.

\medskip

\noindent \textbf{Step 3.} 
It remains to verify that $\L$ induces an equivalence between the fiber over each object in $\Cat(\V)$. To this end, let us fix a categorical algebra $\CC$ with a \emph{set} of objects $S$ and let $\CC^\wedge\in \Cat(\V)$ be its completion. By \cite[Theorem 5.3.17]{GH15}, every $\V$-enriched $\infty$-category arises in this way. 

Let us now describe the fiber of $\Cat(p)\colon \Cat(\T\V)\lrar \Cat(\V)$ over $\CC^\wedge$ in more detail. Note that the left vertical fibration in the pullback square \eqref{d:localization diagram} arises as the Grothendieck construction $\int_{X\in\cS} \Alg_{\Cat}^X(\T\V)\lrar \int_{X\in \cS} \Alg_{\Cat}^X(\V)$. Consequently, there are equivalences
\begin{align*}
\T_{\CC}\Alg_{\Cat}^S(\V)&\simeq \Alg_{\Cat}^S(\T\V)\times_{\Alg_{\Cat}^S(\V)} \{\CC\}\\
&\simeq \Alg_{\Cat}(\T\V)\times_{\Alg_{\Cat}(\V)}\{\CC\}\simeq \Cat(\T\V)\times_{\Cat(\V)} \{\CC^\wedge\}.
\end{align*}
Here the first equivalence is Proposition \ref{p:tangent to algebras}, applied to the operad $\O_S$ whose algebras are $\V$-enriched categorical algebras with object set $S$. The last equivalence follows from the pullback square \eqref{d:localization diagram}. It therefore suffices to verify that the composite functor
$$\xymatrix{
\L_{\CC}\colon \T_{\CC}\Alg_{\O_S}(\V)\simeq \Cat(\T\V)\times_{\Cat(\V)} \{\CC^\wedge\}\ar[r] & \T_{\CC^\wedge}\Cat(\V)
}$$
is an equivalence. Unraveling the definitions, $\L_\CC$ arises from the functor 
$$\xymatrix{
\Alg_{\Cat}^S(\V)\ar[r] & \Alg_{\Cat}(\V)\ar[r]^-{(-)^\wedge} & \Cat(V)
}$$
from $\V$-enriched categorical algebras with object set $S$ to all $\V$-enriched $\infty$-categories, by taking the tangent $\infty$-category at $\CC$. The fact that $\L_{\CC}$ is an equivalence is then exactly \cite[Proposition 3.1.9]{part2}: indeed, we can present the compactly generated monoidal $\infty$-category $\V$ by a monoidal model category $\mathbf{V}$ all of whose objects are cofibrant \cite[Remark 4.1.8.9]{Lur14}, and $\Alg_{\Cat}^S(\V)\lrar \Cat(\V)$ then arises from the obvious functor $\Alg_{\Cat}^S(\mathbf{V})\lrar \Cat(\mathbf{V})$ (which preserves weak equivalences) between model categories of $\mathbf{V}$-enriched categories (see \cite{Hau15}). It was shown in \cite[Proposition 3.1.9]{part2} that this functor of model categories induces an equivalence on tangent $\infty$-categories at each $\CC\in \Alg_{\Cat}^S(\V)$.
\end{proof}
\begin{pro}\label{p:t-structure on tangent of categories}
Let $\V$ be a differentiable presentable SM $\infty$-category such that $1_\V$ is compact. If $\T\V$ carries a monoidal $t$-orientation, then the tangent bundle $\T\Cat(\V)$ carries a monoidal $t$-orientation as well.
\end{pro}
\begin{proof}
Using the equivalence $\T\Cat(\V)\simeq \Cat(\T\V)$ from Proposition \ref{p:tangent to categories}, we define $\T^{\geq 0}\Cat(\V)\simeq \Cat(\T^{\geq 0}\V)$ and $\T^{\leq 0}\Cat(\V)\simeq \Cat(\T^{\leq 0}\V)$. Since $\Cat(-)$ preserves symmetric monoidal functors and fully faithful functors, $\T^{\geq 0}\Cat(\V)\subseteq \T\Cat(\V)$ is closed under the tensor product. 

To see that this determines a $t$-structure, one can repeat the argument given in Step 2 of the proof of Proposition \ref{p:tangent to categories} for $\T^{\geq 0}\V$ and $\T^{\leq 0}\V$ instead of $\T\V$ to show that $\Cat(\T^{\geq 0}\V)\lrar \Cat(\V)$ and $\Cat(\T^{\leq 0}\V)\lrar \Cat(\V)$ are Cartesian fibrations. This verifies condition (1) from Definition \ref{d:t-structure}. Both of these Cartesian fibrations fit into pullback diagrams of the form \eqref{d:localization diagram}. This implies that for a $\V$-enriched $\infty$-category $\CC$ with space of objects $X$, the intersections of $\T^{\geq 0}\Cat(\V)$ and $\T^{\leq 0}\Cat(\V)$ with the fiber over $\CC$ are equivalent to the fibers of
\begin{equation}\label{d:t-structure for cats}\xymatrix{
\Alg_{\Cat}^X(\T^{\geq 0}\V)\hspace{2pt}\ar@{^{(}->}[r] & \Alg_{\Cat}^X(\T\V) & \hspace{2pt}\Alg_{\Cat}^X(\T^{\leq 0}\V)\ar@{_{(}->}[l]
}\end{equation}
over $\CC\in \Alg_{\Cat}^X(\V)$. Since categories with a fixed space $X$ of objects are algebras over a certain $\infty$-operad $\O_X$, it follows from Proposition \ref{p:t-structure on tangent of algebras} that $\T^{\geq 0}\Cat(\V)$ and $\T^{\leq 0}\Cat(\V)$ restrict to a $t$-structure on the fiber over $\CC$.
\end{proof}

\subsection{Local systems of abelian groups on $(\infty, n)$-categories}
Applying Proposition \ref{p:t-structure on tangent of categories} inductively, starting with the $t$-structure on parametrized spectra from Example \ref{e:postnikov for prestable on TS}, we obtain the following:
\begin{cor}
Let $\C$ be an $(\infty, n)$-category. Then the tangent $\infty$-category $\T_{\C} \Cat_{(\infty, n)}$ carries a $t$-structure, in which an object $E$ is (co)connective if and only if for any two objects $x, y\in \C$, the functor 
$$
\Map_{(-)}(x, y)\colon \T_{\C}\Cat_{(\infty, n)}\lrar \T_{\Map_{\C}(x, y)} \Cat_{(\infty, n-1)}
$$
sends $E$ to a (co)connective object.
\end{cor}
The construction of Proposition \ref{p:t-structure on tangent of categories} shows that the heart of the $t$-structure on $\T\Cat(\V)$ can be identified with the Cartesian fibration
$$
\Cat(\T^{\heartsuit}\V)\lrar \Cat(\V).
$$
Applying this inductively, one finds the following inductive description of the heart of the $t$-structure on $\T\Cat_{(\infty, n)}$:
\begin{defn}[{cf.\ \cite[Definition 3.5.10]{Lur09b}}]\label{d:local systems}
The $\infty$-category of \textbf{local systems of abelian groups on $(\infty, 0)$-categories} is defined to be the domain of the Cartesian fibration
$$
\Loc_{(\infty, 0)}\lrar \cS
$$
classified by the functor $\cS^{\op}\lrar \Cat_{\infty}$ sending a space  $X$ to the category of local systems $\Fun\big(\Pi_1(X), \Ab\big)$. This carries a symmetric monoidal structure given by the Cartesian product.

For $n\geq 1$, we define the symmetric monoidal $\infty$-category of \textbf{local systems of abelian groups on $(\infty, n)$-categories} to be the domain of the Cartesian fibration
$$
\Loc_{(\infty, n)} = \Cat\big(\Loc_{(\infty, n-1)}\big)\lrar \Cat\big(\Cat_{(\infty, n-1)}\big)= \Cat_{(\infty, n)}.
$$
Note that $\Loc_{(\infty, n)}$ inherits a symmetric monoidal structure from $\Loc_{(\infty, n-1)}$, such that the projection to $\Cat_{(\infty, n)}$ is symmetric monoidal.
\end{defn}
For each $\C$, let us denote the fiber of $\Loc_{(\infty, n)}$ over $\C$ by $\Loc_{(\infty, n)}(\C)$ and refer to it as the \textbf{abelian category of local systems on} $\C$. Note that $\Loc_{(\infty, n)}(\C)$ is indeed an (ordinary) abelian category, since it can be identified with the fiber of the heart $\T^{\heartsuit}\Cat_{(\infty, n)}\lrar \Cat_{(\infty, n)}$ over $\C$. 
\begin{rem}\label{r:local system}
The proof of Proposition \ref{p:t-structure on tangent of categories} (see diagram \eqref{d:t-structure for cats}) shows that for any $\V$-enriched $\infty$-category $\CC$ with space of objects $X$, there is an equivalence 
$$
\T^\heartsuit_{\CC}\Cat(\V)\simeq \Alg_{\O_X}(\T^{\heartsuit}\V)\times_{\Alg_{\O_X}(\V)} \{\CC\},
$$
where $\O_X$ denotes the operad for categorical algebras with a space of objects $X$. 

Applying this inductively to an $(\infty, n)$-category $\C$ with space of objects $X$, one finds that a local system $\A$ on $\C$ is given by the datum of map of $\infty$-operads, ,
$$\xymatrix{
& \Loc_{(\infty, n-1)}^{\otimes}\ar[d]\\
\O_X\ar[r]_-{\C} \ar@{..>}[ru]^{\A} & \Cat_{(\infty, n-1)}^{\otimes}.
}$$
Note that the right vertical functor is a $1$-categorical fibration, in the sense that the induced maps on mapping spaces have \emph{discrete fibers}. In particular, specifying a lift $\A$ corresponds to choosing images for the maps in $\O_X$ satisfying a certain associativity condition, but \emph{no higher coherences}. In other words, the datum of a local system $\A$ on an $(\infty, n)$-category $\C$ can therefore be described inductively as follows:
\begin{enumerate}[leftmargin=*, label={(\roman*)}]\setlength{\itemsep}{4pt}
\item[(i)] for each $x, y\in \C$, a local system $\A_{x, y}$ over the $(\infty, n-1)$-category of maps $\C(x, y)$.
\item[(ii)] for every triple $x, y, z\in \C$, a map of local systems 
$$
m_{x, y, z}\colon p_1^*\A_{y, z}\times p_0^*\A_{x, y} \lrar c^*\A_{x, z}$$
where $c\colon \C(y, z)\times \C(x, y)\lrar \C(x, z)$ is the composition and $p_0\colon \C(y, z)\times \C(x, y)\lrar \C(y, z)$ and $p_1\colon \C(y, z)\times \C(x, y)\lrar \C(x, y)$ are the projections.
\item[(iii)] for every quadruple $w, x, y, z\in \C$, there is an associativity condition, given by $m_{w, y, z}\circ (\id\times m_{w, x, y})=m_{w, x, z}\circ (m_{x, y, z}\times \id)$.
\end{enumerate}
Definition \ref{d:local systems} therefore gives a precise formulation of the (informal) definition of local systems on $(\infty, n)$-categories appearing in \cite[Definition 3.5.10]{Lur09b}.
\end{rem}

The inductive construction of the abstract Postnikov tower in Theorem \ref{t:main-theorem-oo-n-cats} shows that all parametrized spectra appearing in it are contained in the heart of the $t$-structure on $\T\Cat_{(\infty, n)}$. We therefore obtain the following result (which appears without proof as \cite[Claim 3.5.18]{Lur09b}):
\begin{cor}\label{c:local systems}
For every $(\infty, n)$-category $\C$, the parametrized spectrum $\rH\pi_a(\C)$ of Theorem \ref{t:main-theorem-oo-n-cats} is the Eilenberg--Maclane spectrum associated to
$$
\pi_a(\C)\in \Loc_{(\infty, n)}(\C)=\T^\heartsuit_{\Ho_{(n+1, n)}(\C)}\Cat_{(\infty, n)}.
$$
In terms of Remark \ref{r:local system}, it is the Eilenberg--Maclane spectrum of the local system of abelian groups on $\Ho_{(n+1, n)}\C$ given inductively by $\pi_a(\C)_{x, y} = \pi_a\Map_{\C}(x, y)$, for any $x, y\in \C$.
\end{cor}

\end{document}